\documentclass{amsart}
\usepackage[hang, flushmargin, bottom]{footmisc} % places footnotes at page bottom
\usepackage{enumerate}
\usepackage{enumitem}
\usepackage[hyperindex,breaklinks,pdftex, bookmarks]{hyperref}
\usepackage[color, 
 final  % set "final" option for final version
]{showkeys}
\usepackage[all,  curve, frame, arc, color]{xy}

\usepackage{mathtools}

\usepackage{graphicx, import}        % standard LaTeX graphics tool
\usepackage{multicol}        % used for the two-column index
\usepackage{diagbox}
\usepackage{array,colortbl%,
}
\usepackage{cleveref}
\usepackage{caption} 
\usepackage[utf8x]{inputenc}
\usepackage[english]{babel}
\usepackage{amsfonts}
\usepackage{amssymb}
\usepackage{mathabx}
\usepackage{graphicx}
\usepackage{textcomp}
\usepackage{url}
\usepackage{soul}
\usepackage[textwidth=2cm, textsize=tiny]{todonotes}
\usepackage{verbatim}

\makeatletter
\newcommand{\eqnum}{\refstepcounter{equation}\textup{\tagform@{\theequation}}}
\makeatother

\newtheorem{thm}{Theorem}[section]
\newtheorem{lem}[thm]{Lemma}  %tolto per svmult

\newtheorem{prop}[thm]{Proposition}
\newtheorem{cor}[thm]{Corollary}

\newtheorem{df}{Definition}
\newtheorem{rem}{Remark}

\newcommand{\Z}{\mathbb{Z}}
\newcommand{\C}{{\mathbb{C}}}
\newcommand{\R}{{\mathbb{R}}}
\newcommand{\Q}{{\mathbb{Q}}}
\newcommand{\N}{{\mathbb{N}}}
\newcommand{\F}{\mathbb{F}}

\newcommand{\cB}{\mathcal{B}}

\newcommand{\rk}{{\operatorname{rk}}}

\newcommand{\into}[0]{\hookrightarrow}

\newcommand{\Filt}{\mathcal{F}}

\newcommand{\Ga}[1]{\operatorname{Art}(A_{#1})}
\newcommand{\Gb}[1]{\operatorname{Art}(B_{#1})}

\newcommand{\Gbn}{\operatorname{Art}(B_n)}
\newcommand{\Gdn}{\operatorname{Art}(D_n)}
\newcommand{\Br}{\operatorname{Br}}
\newcommand{\perm}{\mathfrak{S}}
\newcommand{\W}{\operatorname{G}}
\newcommand{\B}{\operatorname{B}}

\newcommand{\disk}{\mathrm{D}}
\newcommand{\ddiskP}{\widetilde{\disk \setminus \P}}
\newcommand{\conf}{\operatorname{C}}
\newcommand{\confn}{\conf_n}
\newcommand{\confm}[1]{\operatorname{C_{1,#1}}}

\newcommand{\confmd}[1]{\widetilde{\operatorname{C_{1,#1}}}}

\newcommand{\sym}[1]{\operatorname{V_{#1}}}

\newcommand{\totsp}{\mathrm{E}}

\renewcommand{\P}{\mathrm{P}}

\newcommand{\surf}{\Sigma}
\renewcommand{\ring}{R}

\renewcommand{\st}{\Gamma}
\newcommand{\pmu}{{\pm 1}}

\newcommand{\ppartial}{\overline{\partial}}
\newcommand{\DB}{\ppartial}

\newcommand{\scsc}[2]{ \substack{
		{}_{#1}\\
		{}_{#2}
	} }

\newcommand{\gp}{\widetilde{\gamma}}
\newcommand{\gpu}{\widetilde{\gamma}_1}

\newcommand{\bool}[1]{\mathcal{P}[#1]}

\DeclareMathOperator{\gs}{\gamma}
\DeclareMathOperator{\gsu}{\gamma_1}
\DeclareMathOperator{\coker}{\mathrm{coker}}

\newcommand{\MMM}{\mathrm{M}}
\newcommand{\MMMp}{\widetilde{\MMM}}
\newcommand{\MM}[1]{\operatorname{\MMM}(#1)}
\newcommand{\MMp}[1]{\MMMp(#1)}
\newcommand{\Id}{\mathrm{Id}}

\newcommand{\stab}{\mathrm{st}}
\newcommand{\gstab}{\mathrm{gst}}

\newlength\Origarrayrulewidth

\newcommand{\chl}{\cellcolor{lightgray!30}}

\def\svgwidth{\columnwidth}

\title[Families of superelliptic curves]{%Homology of Hurwitz spaces\\
Families of superelliptic curves, complex braid groups and generalized Dehn twists}

\author{Filippo Callegaro and Mario Salvetti}
\address[F. Callegaro]{Dipartimento di Matematica, University of Pisa, Italy.}
\email{callegaro@dm.unipi.it}

\address[M. Salvetti]{Dipartimento di Matematica, University of Pisa, Italy.} 
\email{salvetti@dm.unipi.it}

\begin{document}

\begin{abstract}
%\todo{da rivedere}
We consider the universal family $E_n^d$ of superelliptic curves: each curve $\Sigma_n^d$ in the family is a $d$-fold covering of the unit disk, totally ramified over a set $\P$ of $n$ distinct points; $\Sigma_n^d\hookrightarrow E_n^d\to \conf_n$ is a fibre bundle, where $\conf_n$  is the configuration space of $n$ distinct points. \\
We find that $E_n^d$ is the classifying space for the complex braid group of type $\B(d,d,n)$ and we compute a big part of the integral homology of $E_n^d,$ including a complete calculation of the  stable groups over finite fields by means of Poincar\`e series. 
The computation of the main part of the above homology reduces to the computation of the homology of the classical braid group with coefficients in the first homology group of $\Sigma_n^d,$ endowed with the monodromy action.   
While giving a geometric description of such monodromy of the above bundle, we introduce  generalized $\frac{1}{d}$-twists, associated to each standard generator of the braid group, which reduce to standard Dehn twists for $d=2.$ 

%  
%Homology of braid groups and Artin groups can be related to the study of spaces of curves.
%We completely calculate the integral homology of the family of smooth curves of genus $g$ with one boundary component, that are double coverings of the disk ramified over $n = 2g + 1$ points.
%The main part of such homology is described by the homology of the braid group with coefficients in a symplectic representation, namely the braid group $\Br_n$ acts on the first homology group of a genus $g$ surface via Dehn twists. Our computations shows that such groups have only $2$-torsion. We also investigate stabilization properties and provide Poincar\'e series, both for unstable and stable homology.
\end{abstract}

\maketitle
%\todo{abstract da rivedere}

\section{Introduction}
Let
$$
\totsp_{n}^d := \{(\P, z,y ) \in \conf_n  \times \disk \times \C| y^d = (z-x_1)\cdots(z-x_n) \}
$$ 
be 
%n this paper we consider %the particular case of 
the family of superelliptic curves 
where $\disk$ is the  open unit disk in $\C,$   $\conf_n$ is the configuration space of $n$ distinct unordered points in $\disk$ and $\P=\{x_1,\dots,x_n\}\in \conf_n.$  
For a fixed $\P \in \conf_n$, the curve $$\surf^d_n =  \{(z,y ) \in \disk \times \C| y^d = (z-x_1)\cdots(z-x_n) \}$$ in the family $\totsp_n^d$ is a $d$-fold covering of the disk $\disk,$ totally ramified over $\P,$  and there is a fibration  $\pi:\totsp_n^d\to \conf_n$ which takes  $\surf_n^d$ onto its set of ramification points. One can see $\totsp_n^d$ as a universal family over the Hurwitz space $H^{n,d}$ (for precise definitions see \cite{fulton}, \cite{evw}).
%Even though Hurwitz spaces are very classical and studied objects, it seems that their cohomology is not completely known in general. 

In addition to the obvious interest for such families, we find a remarkable fact which seems not having been noticed before (even if the proof is not difficult):  
\begin{thm}[see Theorem \ref{thm:Bddn}]
The space $\totsp^d_n$ is a classifying space for the complex braid group of type $\B(d,d,n)$.
\end{thm}

Here we are interested  in computing the integral homology of the space $\totsp_n^d.$ 
 The rational homology of $\totsp_n^d$ is known, having been computed in \cite{chen} by using \cite{cms_tams}.
The bundle $\pi:\totsp^d_n\to \conf_n$ has a global section, so $H_*(\totsp^d_n)$ splits into a direct sum $H_*(\conf_n)\oplus H_*(\totsp^d_n,\conf_n)$ and by the Serre spectral sequence $H_*(\totsp^d_n,\conf_n)=H_{*-1}(\Br_n;H_1(\surf^d_n)).$ 
We use here that $\conf_n$ is a classifying space for the braid group $\Br_n.$ 
We study the geometric action of the braid group over the surface. Each standard generator of the braid group lifts to a particular homeomorphism of the surface that we call a $\frac{1}{d}$-twist. Such twist is associated to an embedding of a regular polygon with an even number of edges, with opposite edges identified, and having one or two interior holes according to $d$ odd or even (see Figure \ref{fig:poligoni}). Suitable rotations of such polygons induce an homeomorphism of the surface that corresponds to the monodromy action (see \Cref{fig:twistdispari,fig:twistpari}).
Surprisingly, these particular homeomorphisms of a surface have not been considered before. Very recently \cite{KimSong} for $d=3$ and  \cite{GhaMcLeay} for general $d$ independently  found a similar description at the same time as ours.
The $\frac{1}{d}$-twist reduces to a standard Dehn twist around a simple curve for $d=2$   (see \cite{per_van_92}, \cite{waj_99}). 
For $d$ even, the $\frac{d}{2}$-th power of a  $\frac{1}{d}$-twist is a standard Dehn twist: so, we obtain explicit roots of Dehn twists which appear to  have an easier description than those introduced in \cite{marg_schl}.  
In Theorem \ref{thm:action} we describe the induced action on the first homology group of the surface.

In this paper we actually compute the integral homology of the braid group with coefficients in the above representation. Since the homology of the braid groups with trivial coefficients is well-know (see for example \cite{fuks}, \cite{vain}, \cite{cohen}) we obtain a description of the homology of $\totsp^d_n.$ 
%%%%%%%%%%%%%%%%%%%%%
It would be natural to extend the computation to the homology of the braid group $\Br_n$ with coefficients in the symmetric powers of  $H_1(\surf^d_n)$. In the case of $n=3$ and $d=2$ a complete computation (in cohomology) can be found in \cite{ccs}.
%%%%%%%%%%%%%%%%%%%%%

Our main results are the following. For reader convenience, we write again the case $d=2,$ which have already appeared in \cite{cal_sal_superelliptic}.

\begin{thm}[see Theorems \ref{th:no4tor}, \ref{thm:coverings}, \ref{thm:poincare}] \label{thm:homologyintro}

\begin{enumerate}[wide=0pt]
\item[]	
\item For odd $n$ or odd $d$ the homology $H_{i}(\Br_n; H_1(\surf^d_n;\Z))$ has no torsion of order $p^k$ if $p^k \nmid d$. 
\item For odd $n$  and for $p$ prime such that $p\mid d$ 
the rank of $H_i(\Br_n;H_1(\surf_n^d;\Z))\otimes \Z_p$ as a $\Z_p$-module is the coefficient of $q^it^n$ in the expansion of the series
 $$
\widetilde{P}_p(q,t)= \frac{qt^3}{(1-t^2q^2)(1-t^2)} \prod_{j \geq 0} \frac{1+q^{2p^j-1}t^{2p^j}}{1-q^{2p^{j+1}-2}t^{2p^{j+1}}}.
$$
\end{enumerate}
\end{thm}
When $d$ is square-free and $d$ or $n$ is odd Theorem \ref{thm:homologyintro} completely determines the homology groups $H_*(\totsp_{n}^d, \conf_n)$ with integer coefficients.

\begin{thm}[see Theorems \ref{thm:stabilization}, \ref{thm:stablepoincare}]\label{thm:stableintro}
Consider homology with integer coefficients. 
\begin{enumerate}[wide=0pt]
\item The homomorphism 
$$
H_i(\Br_n; H_1(\surf_n^d)) \to H_i(\Br_{n+1}; H_1(\surf_{n+1}^d))
$$
is an epimorphism for $i \leq \frac{n}{2}-1 $
and an isomorphism for $i < \frac{n}{2}-1$.
\item Let $p$ be a prime that does not divide $d$.
For $n$ even the group $H_i(\Br_n; H_1(\surf_n^d))$ has no $p$ torsion 
when $\frac{pi}{p-1}+3  \leq n$ and no free part for $i+3 \leq n$. In particular for $n$ even when $\frac{3i}{2}+3 \leq n$ the group $H_i(\Br_n; H_1(\surf_n^d))$ has only torsion that divides $d$. 
\item 
%The Poincar\'e polynomial of the stable homology $H_i(\Br_n;H_1(\surf_n^d;\Z))\otimes \Z_2$ as a $\Z_2$-module is the following:
%$$
%P_2(\Br;H_1(\Sigma^d))(q) = \frac{q}{1-q^2} \prod_{j \geq 1} \frac{1}{1-q^{2^j-1}}
%$$
%and for $p$ odd 
Let $p$ be a prime that divides $d$. The Poincar\'e polynomial of the stable homology $H_i(\Br_n;H_1(\surf_n^d;\Z))\otimes \Z_p$ as a $\Z_p$-module is the following:
$$
P_p(\Br;H_1(\Sigma^d))(q) = \frac{q}{1-q^2} \prod_{j \geq 0} \frac{1+q^{2p^j-1}}{1-q^{2p^{j+1}-2}}.
$$
\end{enumerate}
\end{thm}
When $d$ is square-free Theorem \ref{thm:stableintro} completely determines the stable homology groups $H_*(\totsp_{n}^d, \conf_n)$ with integer coefficients. 

We also find unstable free components in the top and top$-1$ dimension for $n$ and $d$ both even (see Theorem \ref{thm:unstable}), coherently with the computations in \cite{chen}.

Since we have $H_i(\B(d,d,n)) \simeq H_i(\Br_n) \oplus H_{i-1}(\Br_n;H_1(\surf_n^d;\Z))$ 
and the homology of the braid group with trivial coefficients is classically known (\cite{arnold70}) we get in particular:

\begin{thm}[see Theorem \ref{thm:complexbraidstability}]
The homomorphism $$
H_i(\B(d,d,n)) \to H_{i}(\B(d,d,n+1))
$$
induced by the natural inclusion 
$\B(d,d,n) \into \B(d,d,n+1)$ 
is an epimorphism for $i \leq \frac{n}{2} $
and an isomorphism for $i < \frac{n}{2}$.
\end{thm}

Notice that there were very few cohomological computations about the homology of complex braid groups of type $B(d,d,n);$ in fact, the only known computations (see \cite{calmar}, not more than the second homology groups) used methods based on a resolution given in \cite{dehornoy_lafont} for a Garside monoid introduced in \cite{corranpicantin}. This method seems too complicated to be used for higher homology groups.

%\begin{thm}[(see Theorem \ref{th:no4tor}, \ref{thm:poincare})]
%
%For odd $n:$
%\begin{enumerate}
%\item  the integral homology $H_{i}(\Br_n; H_1(\surf^2_n;\Z))$ has only $2$-torsion. 
%\item the rank of $H_i(\Br_n;H_1(\surf_n;\Z))$ as a $\Z_2$-module is the coefficient of $q^it^n$ in the series
% $$
%\widetilde{P}_2(q,t)=\frac{qt^3}{(1-t^2q^2)} \prod_{i \geq 0} \frac{1}{1-q^{2^i-1}t^{2^i}}
%$$
%In particular the series $\widetilde{P}_2(q,t)$ is 
%the Poincar\'e series of the homology group  %$$\oplus_{n}H_*(\Br_{2n+1};H_1(\surf_{2n+1};\Z))$$ 
%$$\bigoplus_{n \mbox{\scriptsize odd}}H_*(\Br_{n};H_1(\surf^2_n;\Z))$$
%as a $\Z_2$-module. 
%\end{enumerate}
%\end{thm}

%\begin{thm} [(see Theorem \ref{thm:stabilization}, \ref{thm:stablepoincare})]
%Consider homology with integer coefficients. 
%\begin{enumerate}
%\item The homomorphism 
%$$
%H_i(\Br_n; H_1(\surf^2_n)) \to H_i(\Br_{n+1}; H_1(\surf^2_{n+1}))
%$$
%is an epimorphism for $i \leq \frac{n}{2}-1 $
%and an isomorphism for $i < \frac{n}{2}-1$.
%
%\item For $n$ even $H_i(\Br_n; H_1(\surf^2_n))$ has no $p$ torsion (for $p > 2$) when $\frac{pi}{p-1}+3 \leq n$ and no free part for $i+3 \leq n$. In particular for $n$ even,  when $\frac{3i}{2}+3 \leq n$ the group $H_i(\Br_n; H_1(\surf^2_n))$ has only $2$-torsion.
%\item 
%The Poincar\'e polynomial of the stable homology $H_i(\Br_n;H_1(\surf^2_n;\Z))$ as a $\Z_2$-module is the following:
%$$
%P_2(\Br;H_1(\surf^2))(q) = \frac{q}{1-q^2} \prod_{j \geq 1} \frac{1}{1-q^{2^j-1}}
%$$
%\end{enumerate}
%\end{thm}

%%%%%
This paper is a natural continuation of \cite{cal_sal_superelliptic}, where we considered the case $d=2.$ In particular, the main ingredients which here we generalize are: a Mayer-Vietoris geometrical decomposition of the space $E_n^d,$ which allows to reduce the computation to the local homology of some "pieces" which are identified with regular coverings of the configuration space $\confm{n}$ for the Artin group of type $B;$ the adaptation and the use of some of the homology computations given in \cite{calmar} to our case. 
For reader convenience, we collect most of the results we need in Section \ref{sec:homol_artin}.
%\todo{integrato}
Sections \ref{sec:homol_artin}--\ref{sec:no4tor} and \ref{sec:stab}  parallel Sections 3--7 of \cite{cal_sal_superelliptic},  generalizing the results already obtained for $\totsp^2_d$. We will directly refer to \cite{cal_sal_superelliptic} when using the results presented there with no variations. 
%%%%%

Some tables with explicit computations for $d=3,4,5,6$ are provided in the final section.

%The main tools that we use are the following. 
%
%First, we use here some of the geometrical ideas in \cite{bianchi}, where  the author shows that the $H_*(\Br_n;H_1(\surf^2_n))$ is at most $4$-torsion using some exact sequences obtained from a Mayer-Vietoris decomposition. 
%
%Second, we identify the homology groups which appear in the exact sequences with local homology groups of the configuration space $\confm{n}$ of $n+1$ points with one distinguished point. Such spaces are the classifying space of the  Artin groups of type $\mathrm{B},$ so we can use some of  the homology computations given in \cite{calmar}: our results heavily rely on these computations and we collect most of those we need in Section \ref{sec:homol_artin}.
%
%Some explicit computations are provided in Table \ref{tab:conti}.
%%\todo{Inserire risultato finale e teoremi di sezione \ref{sec:stab}}
%%\end{document}
%
%For $d>2$... \todo{completare}

\setcounter{tocdepth}{1}
\tableofcontents

\section{General setting}

We recall that the fundamental group of the configuration space $\confn$ introduced before is the classical braid group 
%on $n$ strands 
$\Ga{n-1}  = \Br_n,$ and that $\conf_n$ is a $K(\Br_n,1)$ (see \cite{fa_neu_62}).  We will make use also of the configuration spaces $\confm{n}$ of $n$ unordered distinct points in $\disk$ with one additional distinct marked point. The fundamental group of $\confm{n}$ is the Artin groups $\Gb{n}$ of type $\mathrm{B}$ and (see for example \cite{bri_73}) the space $\confm{n}$ is a $K(\Gbn,1).$

Fixed  $\P = \{ x_1, \ldots, x_n \}\in \confn$, the set $$\surf_n^d:= \{(z,y) \in  \disk \times \C | y^d = (z-x_1)\cdots(z-x_n) \}$$  is a connected oriented surface with  $\gcd(n,d)$ boundary components. The genus of $\surf_n^d$ is 
$$g = 1 + \frac{(d-1)n-d-\gcd(n,d)}{2},$$
in particular for $d=2$ we have
$g = \frac{n-1}{2}$ for odd $n$ and $g = \frac{n-2}{2}$ for $n$ even.

%Let $p$ be a prime or $0$ and let $\F$ be a field of characteristic $p$. 
%We write $\md$ for the $\ring$-module of Laurent series $\ring[t, t^{-1}] $. 

%We assume $n$ to be an odd integer.

Besides the classical braid group, in what follows we will make large use also of the Artin group $\Gb{n}$ of type $\mathrm{B}$.
%$\Gb{n} = \Gbn$.

%Given an element $\P \in \confn$, we can consider the set of points $$\surf^d_n:= \{(z,y) \in  \disk \times \C | y^d = (z-x_1)\cdots(z-x_n) \}.$$ This is a connected oriented surface with one boundary component for $n$ odd and with two boundary components if $n$ is even. The genus of $\surf_n$ is $g = \frac{n-1}{2}$ for odd $n$ and $g = \frac{n-2}{2}$ for $n$ even.

We consider the family  
$$
\totsp^d_n := \{(\P, z,y ) \in \conf_n  \times \disk \times \C| y^d = (z-x_1)\cdots(z-x_n) \}
$$
which is a fibre bundle with  natural projection $\pi:\totsp^d_n\to \conf_n,$ mapping $(\P, y, z) \mapsto \P$. The fiber of $\pi$ is the surface $\surf^d_n,$ which  is a $d$-fold covering of $\disk,$ totally ramified along $\P$; we will identify $\P$ with a subset of $\surf^d_n$. 

We define $\ddiskP^d:= \surf^d_n \setminus \P$ as the $d$-fold covering of $\disk \setminus \P$ induced by $\surf^d_n \to \disk$.
Since $H_1(\disk \setminus \P)$ has rank $n$ we have that $H_1(\ddiskP^d)$ has rank $d(n-1)+1$. 

The projection 
$\totsp^d_n \stackrel{p}{\longrightarrow} \conf_n \times \disk 
$ given by $p:(\P,z,y) \mapsto (\P, z)$
makes $\totsp^d_n$  a $d$-fold covering of $\conf_n \times \disk,$ which ramifies over  $\confm{n-1}\equiv \{ (\P, z) \in \conf_{n} \times \disk | z \in \P\}$. The complement of
$\confm{n-1} \subset \conf_n \times \disk$ identifies with $\confm{n}$, so the complement \
 $\confmd{n}^d :=\ \totsp^d_n \setminus (p^{-1}(\confm{n-1}))$   is a $d$-fold covering of $\confm{n}.$

%Ramified double covering of $\C \setminus \P$: $\surf_n$. Notice that $\surf_n$ is an open surface of genus $g= (n-1)/2$ and one boundary component and $H_1(\surf_n)$ has rank $2g = n-1$.

%Configurations space of $n$ points in $\disk$ with an additional market point: $\confm{n} = \confmn$.

%Double covering of $\confmn: \confmd{n} = \confmdn$.

%Double covering of $\confmn$ with ramification points: $\totsp_n$.

%Fibration: $\surf_n \to \totsp_n \stackrel{\pi}{\to} \conf_n$. 
\begin{rem}\label{rem:globalsection}
The fibre bundle  $\surf^d_n \into \totsp^d_n \stackrel{\pi}{\to} \conf_n$ admits a global section (see Definition \ref{def:section}) so $H_*(\totsp^d_n) = H_*(\totsp^d_n,\conf_n) \oplus H_*(\conf_n)$ and $$H_i(\totsp^d_n, \conf_n) = H_{i-1}(\conf_n; H_1(\surf^d_n)).$$
\end{rem}

Recall also that the $H_1(\surf^d_n)$  is endowed with an anti-symmetric form given by the cap product. There is a monodromy action of $\pi_1(\conf_n)$ on $\operatorname{Homeo}(\surf^d_n)$  associated to the fibering $\pi;$  the induced action onto $H_1(\surf^d_n)$ preserves this form. We describe this monodromy in more details in the following section. For $d=2$ the monodromy representation maps the standard generators of the braid groups to Dehn twists and is called geometric monodromy (see \cite{per_van_92}, \cite{waj_99}).
Hence we can consider $H_1(\surf^d_n)$ as a $\pi_1(\conf_n)= \Br_n$-representation; we write also $\sym{n,d} :=H_1(\surf_n^d)$.
%, where $g =\frac{n-1}{2}$ for $n$ odd, and $g$ is the genus of

We will need to consider 
%also 
the natural map of the braid group $\Br_n = \Ga{n-1}$ onto the permutation group $\perm_n$ on $n$ letters. This map induces a representation of  the group $\Br_n $ onto $\Z^n$ by permuting cohordinates. We write $\st_n$ for this representation of $\Br_n$.

Along all this paper, when not specified, the homology is understood to be computed with constant coefficients over a ring $\ring$.

\section{Generalized twists}
%\section{A description of the action}
\newcommand{\g}{\gamma}
\newcommand{\tkd}{t_{k,d}}

We give a picture of the surface $\Sigma_n^d$ as a ramified covering of the disk $\disk$ by choosing a system of cuts in $\disk$ connecting the points of $\P$ with the the boundary, as in Figure \ref{fig:disk}. 
\begin{figure}[htb]
	\begin{center}
		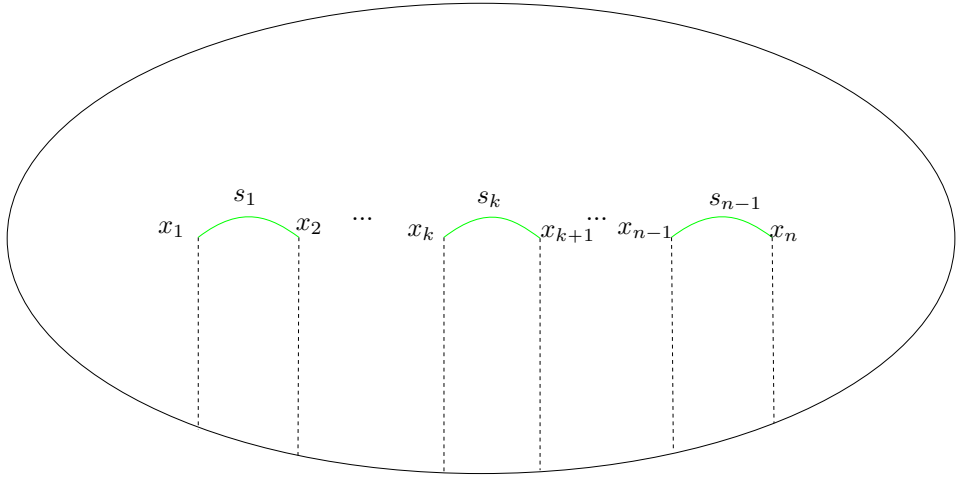
	\end{center}
	\caption{The disk $\disk$ with cuts and arcs.}
	\label{fig:disk}
\end{figure}
The arc $s_k$ connecting $x_k$ with $x_{k+1}$ lifts to the $i$th-sheet to a path $\g^k_i$ which 
connects $x_k$ to $x_{k+1}$ ($i=1,\dots,d;$ \ as said in the previous section, we are identifying
 the $x_k$'s with points in $E_n^d$). Locally in $\surf_{n}^d$, around each ramification point $x_k,$ the sheets 
follow each other in the anticlockwise ordering (see 
%Figures \ref{fig:intorno}, \ref{figillustration3}, \ref{figillustration4}
\Cref{fig:intorno,figillustration3,figillustration4}). 
\begin{figure}[htb]
	\begin{center}
		\def\svgwidth{\columnwidth/2} 
		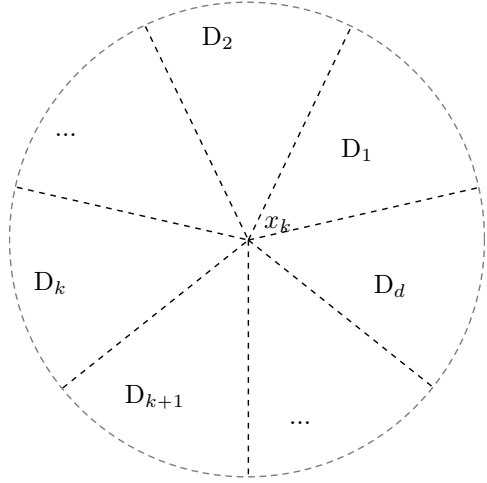
	\end{center}
	\caption{The local picture of sheets following in anticlockwise order around a
 ramification point in $\surf_n^d$.}
	\label{fig:intorno}
\end{figure}
\begin{figure}[htb]
	\begin{center}
		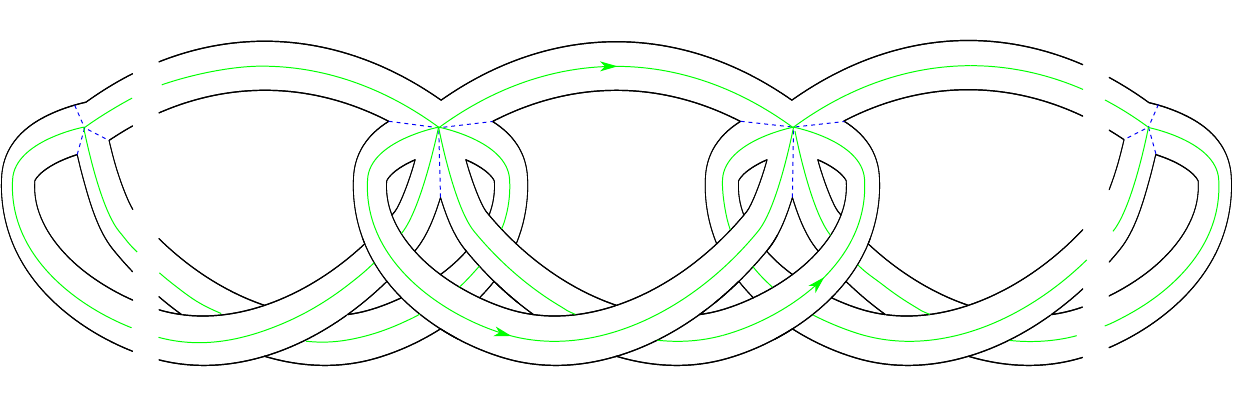
	\end{center}
	\caption{Example of the ramified cover of $\disk$ for $d=3$.}
	\label{figillustration3}
\end{figure}

%	\todo{aggiungere tagli e label, togliere segmenti, cambiare spessore}
\begin{figure}[htb]
	\begin{center}
		\def\svgwidth{\columnwidth} 
		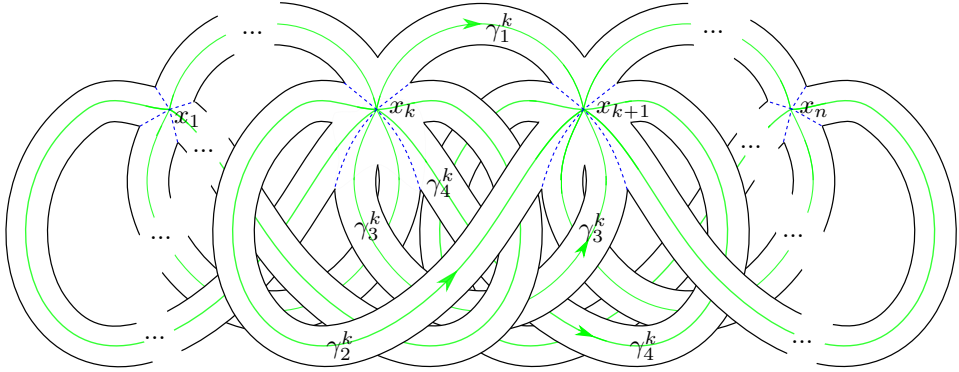
	\end{center}
	\caption{Example of the ramified cover of $\disk$ for $d=4$.}
	\label{figillustration4}
\end{figure}
Clearly, $\Sigma_n^d$ deformation retracts onto the graph with set of vertices $\P=\{x_1,\dots,x_n\}$ and edges $\g^k_i,\ k=1,\dots,n-1,\ i=1,\dots, d,$ therefore $b_1(\Sigma_n^d)=(n-1)(d-1).$ Let us give to $\g^k_i$ the orientation going from $x_k$ to $x_{k+1}.$ We can consider the circuits $a^k_i=\g^k_i(\g^k_{i+1})^{-1},$ $i=1,\dots,d$ (taking indices mod $d$); notice the relations $\sum_{i=1}^d[a_i^k]=0$ among their classes in $H_1$. A basis for $H_1(\Sigma_n^d)$ is given by the classes $[a^k_i],$ \ $k=1,\dots,n-1,\ i=1,\dots,d-1.$ 

Next, we give a precise description of the monodromy action for the bundle $\surf^d_n \into \totsp^d_n \stackrel{\pi}{\to} \conf_n.$  

Let $\sigma_k$ be the  standard generator for the braid group $\Br_n,$ given by an anti-clockwise
 half-twist around the arc $s_k,$  exchanging $x_k$ and $x_{k+1}$ and leaving everything outside 
a neighborhood of $s_k$ pointwise fixed. We find that  $\sigma_k$ lifts to an homeomorphism 
$t_{k,d}$ of $\Sigma_n^d$ which generalizes standard Dehn twist for $d=2$.  

Let $\Gamma_k=\cup_{i=1,\dots,d}\ \g^k_i;$ first, $\tkd$ is the identity outside a small 
neighborhood of $\Gamma_k.$  Inside a small neighborhood of $\Gamma_k,$  $t_{k,d}$ acts as 
``rotation around $\Gamma_k$ of a $\frac{2\pi}{d}$-angle." To be precise,   we describe the 
neighborhood of $\Gamma_k$ as a $2-$complex.

Let $Q=[0,1]\times [0,1]\subset \R^2,$ with coordinates $(x,y)$ and the standard $CW$-structure. We orient the horizontal edges according to increasing $x$-coordinate.
\begin{df}\label{polyneighbor}  
 For $i=1,\dots,d,$ $j=1,-1,$ let $Q_{i,j}$ be a copy of $Q,$ with coordinates $(x,y)_{i,j}.$ We define the $2$-complex $\tilde{T}_d$ as 
 $$\tilde{T}_d:=\ 
\left.
\left(\bigsqcup_{\substack{i=1,\dots,d\\ j=1,-1}}\ Q_{i,j}\right)
\right/
\left(\begin{array}{c} (1,y)_{i,1}\sim (1,y)_{i+1,-1} \\ (0,y)_{i,1}\sim (0,y)_{i-1,-1} \end{array}\right)
$$     
where the indices are considered $mod\ d.$

Now define
$$T_d\ :=\ \tilde{T}_d\ / \ {(x,0)_{i,1}\sim (x,0)_{i,-1} } \qquad\quad (i=1,\dots,d). $$
%($i=1,\dots,d $).
\end{df}      

Notice that 
$\tilde{T}_d$ is connected for odd 
$d,$ homeomorphic to a cylinder, while it has two connected components for even $d,$ both 
homeomorphic to a cylinder. 
Therefore  $T_d$ is an orientable surface with one boundary component for $d$ odd and two 
boundary components for $d$ even. 
In case $d$ even, make a preliminary step by attaching in $\tilde{T}_d$ the two edges corresponding to the index $i=1:$\  $(x,0)_{1,1}\sim (x,0)_{1,-1},\ x\in[0,1].$ Then $T_d$  is obtained from a regular polygon $\mathcal P_d$ having $2d$ edges and one hole for odd $d,$ having $2(d-1)$ edges and two holes for even $d,$ and attaching the opposite edges, as in \Cref{fig:poligoni}.  One verifies that the genus of $T_d$ is  $\left[\frac{d-1}{2}\right].$ 

\begin{figure}[htb]
	\begin{center}
		\def\svgwidth{\columnwidth} 
		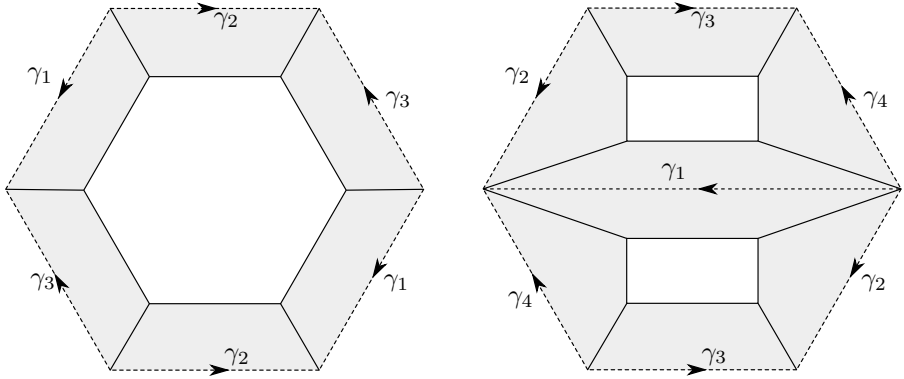
	\end{center}
	\caption{Examples of $\mathcal P_d$ for $d=3,4$. $T_d$ is obtained by attaching the opposite edges $\g_i$ according to the arrows.}
	\label{fig:poligoni}
\end{figure}

The image in $T_d$ of the set $\Gamma:=\bigsqcup_{i,j} \{([0,1]\times \{0\})_{i,j}\}$ is homeomorphic to the graph $\Gamma_k$ (where the vertices correspond to  the images in $T_d$ of $\bigsqcup_{i,j} \{(\{0,1\}\times \{0\})_{i,j}$).
Actually, there is an embedding  $j_k:T_d\to\Sigma_n^d$ 
taking $\Gamma$ to $\Gamma_k$ and $T_d$
to a small neighborhood of $\Gamma_k$ in $\Sigma_k $ 
(see \Cref{figillustration3,figillustration4}). 
%\todo{figura incollamento anello/poligono su figura con manici tra due punti ($d$ pari e dispari o solo dispari)}

\begin{df}\label{dtwist} For $h\in\R,\ 0\leq h\leq 1,$ we define the \emph{$h$-rotation} $\rho_h$ in $T_d$ as the map defined by:
$$\begin{array}{lcl} \rho_h((x,y)_{i,1}) &= &\left\{\begin{array}{ll} (x+h, y)_{i,1} & \mbox{if \ }x+h\leq 1\\ (2-x-h,y)_{i+1,-1} & \mbox{if \ } x+h \geq 1 \end{array} \right.; \\   
 & & \\
\rho_h((x,y)_{i,-1})& = &\left\{\begin{array}{ll} (x-h, y)_{i,-1} & \mbox{if \ } x-h\geq 0\\ (h-x,y)_{i+1,1} & \mbox{if \ } x-h \leq 0 \end{array} \right.  . \end{array}$$   
\bigskip

$\left(\mbox{ in one formula:\ 
$\rho_h((x,y)_{i,j}) = \left\{\begin{array}{lc} (x+jh, y)_{i,j} & \mbox{if \ }0\leq x+h\leq 1\\ (1-x+j(1-h),y)_{i+1,-j} & \mbox{otherwise  } \end{array}\right.$  }\right)$   
\bigskip

Let $\varphi:[0,1]\to[0,1]$ be a $C^{\infty}$ function such that: 
\begin{itemize}
	\item[i)] $\varphi(t)=1$ for $t<\epsilon;$ \   $\varphi(t)=0$ for $t>1-\epsilon$  ($\epsilon<<1$); 
	\item[ii)] $\varphi$  is decreasing in $\epsilon<t<1-\epsilon.$ 
\end{itemize}
We define the $\frac{1}{d}-twist$ on $T_d$ as the homeomorphism 
$$\tau_d((x,y)_{i,j})\ =\ \rho_{\varphi(y)}((x,y)_{i,j}).$$
\end{df}

\begin{figure}[htb]
	\begin{center}
		\def\svgwidth{\columnwidth} 
		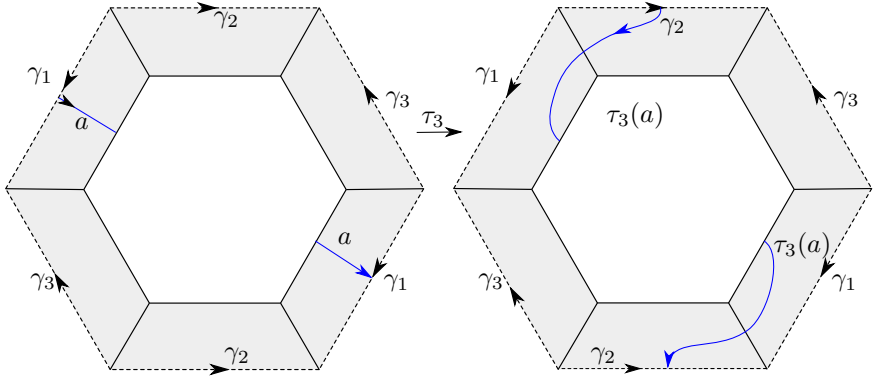
	\end{center}
	\caption{The twist $\tau_3$.}
	\label{fig:twistdispari}
\end{figure}
\begin{figure}[htb]
	\begin{center}
		\def\svgwidth{\columnwidth} 
		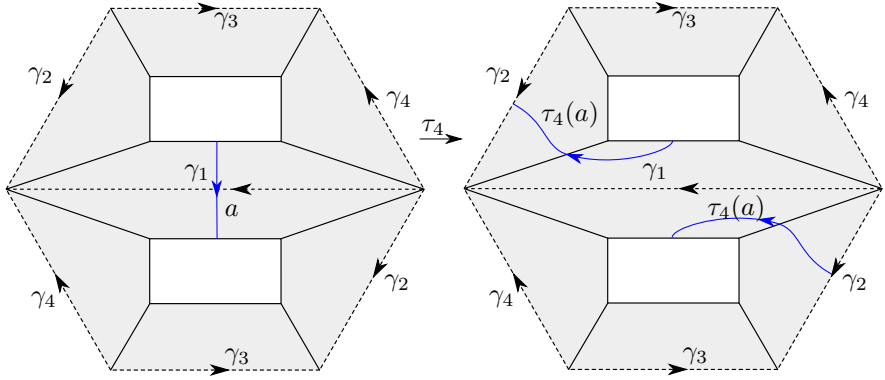
	\end{center}
	\caption{The twist $\tau_4$.}
	\label{fig:twistpari}
\end{figure}

In \Cref{fig:twistdispari,fig:twistpari} we represent the action of $\tau_d$ by using the polygon $\mathcal P_d.$ For odd $d,$ $\tau_d$ is the identity in the boundary of $T_d$ and is induced by a rotation of $\frac{\pi}{d}$  in the exterior boundary of $\mathcal P_d,$ in the sense taking $\g_i$ into $\g_{i+1},\ i=1,\dots,d.$  For even $d,$ $\tau_d$ is the identity in the boundary of $T_d$ (corresponding to the boundary of the two holes of $\mathcal P_d$) and is induced by a rotation of $\frac{2\pi}{d}$, one for each of the two halves of $\mathcal P_d$ determined by $\g_1,$ in the sense which takes $\g_i$ into $\g_{i+1},\ i=1,\dots,d.$

\begin{df}[generalized Dehn twist]\label{dehntwist} The corresponding $\frac{1}{d}-twist$ in the surface $\Sigma_n^d$ is the homeomorphism $\tkd$ induced by $\tau_d$ through the embedding $j_k.$\end{df}

\noindent In particular, $\tkd$ takes $\g^k_i$ into $(\g^k_{i+1})^{-1}$ (indices mod $d$). 

\begin{rem} Notice that for $d$ even $(t_{k,d})^{\frac{d}{2}}$ is a standard Dehn twist around a simple curve following the path $\prod_{i=1}^d (\g^k_i)^{(-1)^{i-1}}.$ So, we obtain an explicit root of a Dehn twist which appears to be simpler than the one described in \cite{marg_schl}.

\end{rem}

We denote by $(\ ,\ )$ the intersection product 
$$(\ ,\ ):H_1(\Sigma^d_n;\Z)\times H_1(\Sigma^d_n,\P_n;\Z)\to \Z$$
By the exact sequence
$$0\to H_1(\Sigma^d_n;\Z) \to H_1(\Sigma^d_n,\P_n;\Z) \to H_0(\P_n;\Z)=\Z^n\to H_0(\Sigma^d_n;\Z)
=\Z\to 0$$
it follows that $H_1(\Sigma^d_n,\P_n;\Z)$ is free of rank $(n-1)d,$ 
generated by the classes $[\g^k_i],\ k=1,\dots,n-1,\ i=1,\dots,d.$

It is easy to verify the following intersection products (we consider indices $i, j$ mod $d$):

%\begin{equation}\begin{array}{ll} ([a^k_i],[a^k_{i+1}])=1 & k=1,\dots, n-1,\ i=1\dots d \\[3pt]
%([a^k_i],[a^k_{j}])=0 & |i-j|>1 \\[3pt]
%([a^k_i],[a^{k+1}_{i}])=-1,\ ([a^k_i],[a^{k+1}_{i+1}])=1  & k=1,\dots, n-2,\ i=1,\dots, d \\[3pt] 
%([a^k_i],[a^{k+1}_{j}])=0 & j-i \not= 0, \ 1\\[3pt]
%([a^k_i],[\g^{k}_{i}])=1,\ ([a^k_i],[\g^{k}_{i+1}])=-1 & k=1,\dots, n-1,\ i=1,\dots, d \\[3pt]
%([a^k_i],[\g^{k}_{j}])=0 & j\not= i, i+1\\[3pt]
%([a^k_i],[\g^{k+1}_{j}])=0
%\end{array}\end{equation}

\begin{equation}\label{interproducts}\begin{array}{lcl} ([a^k_i],[a^k_{j}])&= &\left\{ 
\begin{array}{lcl}  1 &\mbox{if}& j=i+1\\
-1 &\mbox{if}& j=i-1\\
0 & \mbox{otherwise} & \end{array}\right. \\[10pt]
([a^k_i],[a^{k+1}_{j}])&= &\left\{ \begin{array}{lcl}  -1 &\mbox{if}& j=i\\
1 & \mbox{if} & j=i+1 \\
0 & \mbox{otherwise} & \end{array}\right. \\[15pt]
([a^k_i],[\g^{k}_{j}])&= &\left\{ \begin{array}{lcl}  1 &\mbox{if}& j=i\\
-1 & \mbox{if} & j=i+1 \\
0 & \mbox{otherwise} & \end{array}\right. \\[15pt]
([a^k_i],[\g^{k+1}_{j}])&= &\left\{ \begin{array}{lcl}  
-1 & \mbox{if} & j=i+1 \\
0 & \mbox{otherwise} & \end{array}\right. \\[15pt]
([a^{k+1}_i],[\g^{k}_{j}])&= &\left\{ \begin{array}{lcl} 
1 & \mbox{if} & j=i \\
0 & \mbox{otherwise} & \end{array}\right. \\[15pt]

\end{array}\end{equation}

\begin{thm} \label{thm:action}
\begin{enumerate}
\item[]
\item[(a)] 
The $\frac{1}{d}-twist$ $\tkd$ is the homeomorphism induced by the monodromy of the bundle $\Sigma_n^d\hookrightarrow E^d_n\to \C_n$ applied to the half-twist $\sigma_k\in\Br_n.$ 
\item[(b)] 
$\tkd$ induces on $H_1(\Sigma^d_n;\Z)$ the automorphism
\begin{equation}\label{h1action}(\tkd)_*:\ a \to a-\sum_{i=1}^{d}\ (a,\g^k_i)[a^k_i]\end{equation}
\item[(c)]\label{thm:actionc}
Let $t_{a^k_i}$ be the Dehn twist associated to the simple curve $a^k_i,$ $i=1,\dots,d-1;$ then 
\begin{equation}\label{dehntwists}(\tkd)_*\ =\ (t_{a^k_1}\dots t_{a^k_{d-1}})_*\end{equation}
\end{enumerate}
\end{thm}

\begin{proof}  Statement (a) follows by direct verification that  $t_{k,d}$ projects 
through $\pi$ into the half-twist  $\sigma_k.$

Let the class of $a$ be represented by a cycle $\tilde{a}$ which intersects transversely
every $\g^k_i.$ For every point $p\in \tilde{a}\cap \g^k_i,$ \ $t_{k,d}$ modifies $\tilde{a}$  
by adding a cycle in the class of $[a^k_i],$ according to the sign of the intersection 
(see \Cref{fig:twistdispari,fig:twistpari}). 
Therefore (b) follows. 

 We prove (c) by induction on $d.$  First, 
 recall that a Dehn twist around a simple curve $c$ acts on the first homology group by 
 $a\to a-(a,[c]) [c].$  

 We have 
 $$(t_{k,d-1})_*(a)=\ a - \sum_{i=1}^{d-2}\ (a,\g^k_i)[a^k_i] - (a,\g^k_{d-1})[\tilde{a}^k_{d-1}]$$
 where $[\tilde{a}^k_{d-1}]=[\g^k_{d-1}(\g^k_1)^{-1}]=[a^k_{d-1}]+[a^k_d].$ Therefore one has
 $$t_{k,d}(a)= t_{k,d-1}(a)+(a,[a^k_{d-1}])[a^k_d]=(t_{a_1}\dots t_{a_{d-2}})_*(a)+(a,[a^k_{d-1}])[a^k_d]$$
where the last equality comes from induction. But from equations \eqref{interproducts} and  $\sum_{i=1}^d [a^k_i]=0$  it easily follows
$$(t_{a_1}\dots t_{a_{d-2}})_*([a^k_{d-1}]) = -[a^k_d]$$
therefore
$$t_{k,d}(a)=(t_{a_1}\dots t_{a_{d-2}})_*(a-(a,[a^k_{d-1}])[a^k_{d-1}])\ =\ 
(t_{a_1}\dots t_{a_{d-1}})_*(a).$$

\end{proof}

Part (a) and (c) of the preceding theorem has been proved independently for $d=3$ in \cite[Thm.~4.1]{KimSong} and for any $d$ in  \cite[Prop.~5.3]{GhaMcLeay}, where the $\frac{1}{d}-twist$ are called ``chain twists''.

\section{Motivations and homology of complex braid groups}
We recall that the classification of irreducible complex  reflection groups is given in \cite{shto54}. 
The group $\W(de,d,n)$ is the group of monomial matrices such that all the non-zero entries are $(de)$-th roots of unity and the product of all the non-zero entries is a $e$-th root of unity. 
In particular in this paper we are interested in the complex reflection groups of type $\W(d,d,n),$ that 
act naturally as complex reflection groups on the space $\C^n$.
The complement of the reflection arrangement associated to the group  $\W(d,d,n)$ is 
$$
\mathcal{M}(d,n):= \{ (z_1, \ldots, z_n) \mid  \forall i <j, a \in \Z, z_i \neq \zeta_d^a z_j  \}.
$$
and the group $\W(d,d,n)$ acts freely on this space (see for example \cite{naka83, brmaro98}).
For $d=2$ the group $\W(d,d,n)$ is the Coxeter reflection group of type $D_n$.

\begin{thm}\label{thm:Bddn}
The space $\totsp^d_n$ is homotopy equivalent to the regular orbit space of the complex reflection group of type $\W(d,d,n)$. Hence it is a classifying space for the complex braid group of type $\B(d,d,n)$.
\end{thm}

%As noticed in \cite{per_van_96}, the result is already known 

Recall that for $d=2$ we have that the group $\B(2,2,n)$ is the Artin group $\Gdn$. In this case the monodromy action of the braid group $\Br_n$ on the fundamental group of the fiber $\surf_n^d$  is described in \cite{cp05} (see also \cite{per_van_96}).

\begin{proof}
For this proof it will be convenient to replace the space $\totsp_n^d$ with the homotopy equivalent space 
$$\tilde\totsp_n^d:=\{(P, z, y) \in \mathrm{Conf}_n(\C) \times \C \times \C \mid y^d  = (z - x_1) \cdots (z-x_n)\}.$$
Then we define the space
$\bar \totsp^d_n$ that can be identified with the subset of $\tilde\totsp^d_n$ given by the points $(\P, z,y )$ with $z=0$. 
We set
$$
\bar \totsp^d_n := \{(\P, y ) \in \mathrm{Conf}_n(\C)  \times \C| y^d = (-1)^n x_1\cdots x_n \}
$$
where, as usual, we write $P$ for the unordered configuration of points $\{x_1, \ldots, x_n\} \subset \C$.
Next we can define a retraction 
%of the space $\tilde\totsp^d_n$ to the subset $\bar \totsp^d_n$ 
which is an homotopy equivalence:
let $\rho: \tilde\totsp^d_n \to \bar \totsp^d_n$ be the map
$$
(\P, z,y ) \mapsto \left(\P-z,y \right)
$$ 
where, given the set $\P = \{x_1, \ldots, x_n\} \subset \C$, we write $\P-z$ for the set $\{x_1 - z, \ldots, x_n - z\}$.
Clearly $\rho$ is an homotopy equivalence since its fiber is contractible. 

Finally we consider the regular orbit space of the complex reflection group $\W(d,d,n)$. 
We can define a continuous map 
$$
\Xi:\mathcal{M}(d,n)\to \bar \totsp^d_n
$$
by
$$
\Xi:(z_1, \ldots, z_n ) \mapsto (\{z_1^d, \ldots, z_n^d\},\zeta_{2d}^n z_1\cdots z_n)
$$
where $\zeta_{2d}$ is a $2d$-th primitive rooth of unity.
It is easy to verify that the map $\Xi$ is $\W(d,d,n)$-invariant, hence it induces a map on the quotient space
$$
\bar\Xi:\mathcal{M}(d,n) / \W(d,d,n) \to \bar \totsp^d_n.
$$
%by
%$$
%\Xi:[(z_1, \ldots, z_n )] \mapsto (\{z_1^d, \ldots, z_n^d\}, z_1\cdots z_n).
%$$
It is straightforward to check that the $\Xi$-fiber of a point $X \in \bar \totsp^d_n$ is the $\W(d,d,n)$-orbit of a point $\widetilde{X}$ in the counter-image of $X$. This implies that  the map $\bar\Xi$ is a bijection.

Moreover, since the map $\bar\Xi$ is bijective, proper 
and closed, % \todo{cfr Looijenga, Isolated singular points on complete intersections, CUP 1984} 
it induces an homeomorphism between the two spaces (see for example \cite[Prop.~1.11]{loo84} for a detailed argument).
%Argomento alternativo per omeomorfismo:
%-stessi degree (cfr. unitary ref. grp.) quindi generano alg. di funz. inv. e  dunque $\Xi$ \`e orbit map;
%-il suo jacobiano, che \`e lo Jacobiano dell'orbit map, ha sempre rango massimo fuori dall'arrangiamento, quindi la mappa $\Xi$ \`e un omeomorfismo.

From the fibration
$$
\Sigma_n^d \into \totsp_n^d \to \conf_n
$$
where the left and the right term are $k(\pi,1)$ spaces, we obtain that $\totsp_n^d$ is a $k(\pi,1)$ space for $\pi = \W(d,d,n)$.
\end{proof}

Since the surface $\surf_n^d$ is a $k(\pi,1)$ space, our construction gives a short proof of the result in \cite{naka83} for the complex arrangements of type $G(d,d,n)$:
\begin{cor}
The space $\mathcal{M}(d,n)$ is a $k(\pi,1)$.
\end{cor}

\begin{rem}
	Let us consider the subset of $\C^n$
$$
\mathcal{M}_0(d,n):= \{ (z_1, \ldots, z_n) \mid  \forall i, z_i \neq 0, \forall i <j, a \in \Z, z_i \neq \zeta_d^a z_j  \}.
$$
with the natural free action of the complex reflection group $\W(d,1,n)$.
Since the map $\Xi_0: \mathcal{M}_0(d,n) / \W(d,1,n) \to \mathrm{Conf}_n(\C^*)$ is a covering, we obtain that $\mathcal{M}_0(d,n)$ is a $k(\pi,1)$ space as already proved in \cite{naka83}.
\end{rem}

\section{Homology of some Artin groups}\label{sec:homol_artin}

We use the notations and the technical results collected in \cite[\S~3]{cal_sal_superelliptic} for the case $d=2.$ %where we introduced some notations and collect some technical results
% from \cite{calmar} 
%concerning the homology of $\Ga{n}$ and $\Gb{n}$ with constant coefficients and with coefficients in abelian local systems.
%We 
%following \cite{calmar}. 
We indicate here all the generalizations that we need for general $d.$

From the description given in \cite[Thm.~4.5, Thm.~4.12, Rmk.~4.13]{calmar} we have the following results.

\begin{prop}\label{prop:homol_p2}
Let $\F$ be a field of characteristic $p$.
%\todo{completare per $p$ dispari}
For $p$ an odd prime and $n$ odd the  $\F[t]$-module $\oplus_{i,n} H_i(\Gb{n}; \F[t^\pmu])$ has a basis given by
\begin{equation}\label{generators_p1}
%\gs(z_c, x_0 x_{i_1} \cdots x_{i_k}):=
\frac{\DB(z_{2m+1}h^{r-1} y_{j_1} \cdots y_{j_l }x_{i_1} \cdots x_{i_k})}{(1+t)}.
\end{equation}
where $r>0$, $j_1 \leq \cdots \leq  j_l$, $ i_1 < \cdots < i_k$, these  generators have
torsion of order $(1+t)$.
\end{prop}

We can compute the homology groups   $H_*(\Gb{n}; \F_p[t]/(1-(-t)^d))$ 
%and the Bockstein homomorphism $\beta_2$ 
by using the explicit description of \cite[\S~4.3, 4.4, 4.5 and 4.6]{calmar}.
As a special case of \cite[Prop.~4.7, 4.14]{calmar} we have 
%\begin{prop}
the isomorphism
\begin{equation}\label{eq:decomposizione}
H_i(\Gb{n}; \F_p[t]/(1-(-t)^d)) = h_{i}(n,p) \oplus h_{i}'(n,p)
\end{equation}
where the two summands are determined by the following exact sequence:
$$
0 \to h_{i+1}'(n,p) \to H_i(\Gb{n}; \F_p[t^{\pmu}]) \stackrel{(1-(-t)^d)}{\longrightarrow} H_i(\Gb{n}; \F_p[t^{\pmu}]) \to h_{i}(n,p) \to 0.
$$
%\end{prop}
For odd $n$ all the elements of $H_i(\Gb{n}; \F_p[t^{\pmu}])$ are multiple of $x_0$ for $p=2$ and multiple of $h$ for odd $p$; hence they have $(1+t)$-torsion (see Proposition \ref{prop:homol_p2}). This implies that  the multiplication by 
$(1-(-t)^d)$ is the zero map and the generators of  $h_{i}'(n,p)$ and $h_{i}(n,p)$ are  in bijection with a set of generators of $H_i(\Gb{n}; \F_p[t^{\pmu}])$.

As in  \cite[\S~3, prop. 3.5]{cal_sal_superelliptic} for $d=2,$  we have :
\begin{prop} \label{prop:generators}
%	\todo{adattare al caso $p$ dispari e $d>2$, magari solo $p\mid d$ se serve.}
For odd $n$ the homology $H_*(\Gb{n}; \F_2[t]/(1-(-t)^d))$ is generated, as an $\F_2[t]$-module, by the classes of the form 
\begin{equation}\label{generators1}
\gp (z_c, x_0 x_{i_1} \cdots x_{i_k}) := \frac{1-(-t)^d}{1+t} z_{c+1} x_{i_1} \cdots x_{i_k}.
\end{equation}
that correspond to the generators of $h_{i}'(n,2)$
and
\begin{equation}\label{generators2}
\gs(z_c, x_0 x_{i_1} \cdots x_{i_k}):=\frac{\DB(z_{c+1}x_{i_1} \cdots x_{i_k})}{(1+t)}.
\end{equation}
that correspond to generators of $h_{i}(n,2)$. Here $0 \leq i_1 \leq \cdots i_k$, $c$ is even and both kind of generators have
torsion of order $(1+t)$.

For odd $n$ and $p$ an odd prime the homology $H_*(\Gb{n}; \F_p[t]/(1-(-t)^d))$ is generated, as an $\F_p[t]$-module, by the classes of the form 
\begin{equation}\label{generators1p}
	\gp (z_c, h^{r} y_{j_1} \cdots y_{j_l } x_{i_1} \cdots x_{i_k}) := \frac{1-(-t)^d}{1+t}  z_{c+1} h^{r-1} y_{j_1} \cdots y_{j_l }x_{i_1} \cdots x_{i_k}.
\end{equation}
that correspond to the generators of $h_{i}'(n,p)$
and
\begin{equation}\label{generators2p}
	\gs(z_c, h^{r} y_{j_1} \cdots y_{j_l } x_{i_1} \cdots x_{i_k}):=\frac{\DB(z_{c+1}h^{r-1} y_{j_1} \cdots y_{j_l }x_{i_1} \cdots x_{i_k})}{(1+t)}.
\end{equation}
that correspond to generators of $h_{i}(n,p)$. Here $0<r$, $j_1 \leq \cdots \leq  j_l$, $ i_1 < \cdots < i_k$, $c$ is even and both kind of generators have
torsion of order $(1+t)$.
\end{prop}

%%
%% Sposto da qui
%%
%%	
%\todo{sistemare }

%We do not provide a description of the generators of $H_i(\Gb{n}; \F[t^\pmu])$ for a field $\F$ of characteristic $p>2$ for $n$ even and we also avoid a detailed presentation of a set of generators of  $H_i(\Gb{n}; \F[t]/(1-(-t)^d))$ for a generic field $\F$ when $n$ is even. For such a description, we refer to \cite{calmar}.

Generalizing the sets introduced in \cite[\S 3]{cal_sal_superelliptic}, we provide sets of elements $\cB'$, $\cB''$ of the $\Z$-modules $H_i(\confm{n};\Z)\simeq H_i(\Gb{n}; \Z[t]/(1+t))$ and $H_i(\confmd{n}^d;\Z)\simeq H_i(\Gb{n}; \Z[t]/(1-(-t)^d))$ for $n$ odd, such that the following two condition are satisfied:
\begin{enumerate}[label=(\roman*), wide=0pt]
	\item $\cB'$ (resp.~$\cB''$) induces a base of the homology of $H_i(\confm{n};\Q) $ (resp.~$H_i(\confmd{n}^d;\Q)$);
	%\item the elements in $\cB'$ (resp.~$\cB''$) generate the $2$-torsion submodule in $H_i(\confm{n};\Z)$ (resp.~$H_i(\confmd{n};\Z)$) with integer coefficients;
	\item the images of the elements of $\cB'$ (resp.~$\cB''$) in $H_i(\confm{n};\Z_p)$ (resp.~$H_i(\confmd{n}^d;\Z_p)$) are
	linearly independent for any prime $p$.
\end{enumerate}
\begin{df}\label{def:BB}%\todo{adattare a $\confmd{n}^d$}
Let $n$ be an odd integer. We define the sets $\cB' \subset H_i(\confm{n};\Z)$ (for $d=1$) and $\cB'' \subset H_i(\confmd{n}^d;\Z)$ (for $d>1$) given by the following elements:
$$ 
\omega_{2i,j,0}^{(d)}:= \frac{\DB(z_{2i+1}x_0^{j-1})}{(1+t)} \mbox{ and } \widetilde{\omega}_{2i,j,0}^{(d)}:= \frac{(1-(-t)^d)z_{2i+1}x_0^{j-1}}{(1+t)} \qquad \mbox{ for }j>0;
$$
and
$$
\omega_{2i,j,1}^{(d)}:=\frac{\DB(z_{2i+1}x_0^{j-1}x_1)}{(1+t)} \mbox{ and } \widetilde{\omega}_{2i,j,1}^{(d)}:=\frac{(1-(-t)^d)z_{2i+1}x_0^{j-1}x_1}{(1+t)} \qquad \mbox{ for }j>0.
$$
\end{df}
%where for $e=1$ we have elements in $\cB'_1$, while for $e=2$ we have elements in $\cB''_1$.
The  elements above give a basis for $H_*(\confm{n};\Q) $ (resp.~$H_*(\confmd{n}^d;\Q)$) for $n$ odd (condition (i)) (see \cite[\S4.2]{calmar}). 
%\todo{spiegare richiamando anche Prop.~\ref{prop:hrazionale}}

For condition (ii), one verifies that the elements in $\cB'$ and $\cB''$ define, mod $2$, a subset of the bases of  $H_i(\Gb{n}; \Z_2[t]/(1+t))$ and $H_i(\Gb{n}; \Z_2[t]/(1-(-t)^d))$ given in \cite[\S4.4]{calmar} and, mod $p$ for an odd prime, a subset of the bases of $H_i(\Gb{n}; \Z_p[t]/(1+t))$ and $H_i(\Gb{n}; \Z_p[t]/(1-(-t)^d))$  given in \cite[\S 4.6]{calmar}.

So we have:
\begin{prop}\label{prop:BB}%\todo{adattare a $\confmd{n}^d$}
For $n$ odd the elements of $\cB'$ (resp.~$\cB''$) are a free set of generators of a maximal free $\Z$-submodule of $H_i(\confm{n};\Z)$ (resp.~$H_i(\confmd{n}^d;\Z)$).
\end{prop}

\begin{prop}\label{prop:hrazionale_minus_t}
Let $n$ be an even integer. Let $\F$ be a field of characteristic $0$.
%The Poincar\'e polynomial of 
%$H_*(\confm{n};\F) = H_*(\Gb{n}; \F[t]/(1+t))$ is $(1+q)(1+q+\ldots+q^{n-1})$
%and a base of the homology is given by the following generators
%$$ \frac{\DB (z_{2i+1}x_0^{j-1}x_1)}{1+t}, z_{2i+1}x_0^{j-1}x_1, \frac{\DB (z_{2i+2})}{1+t},  z_{2i+2}.$$
The Poincar\'e polynomial of 
$H_*(\Gb{n}; \F[t]/([d]_{(-t)}))$ %\todo{estendere al caso a coefficienti in $\Q[t^{\pm 1}]/[d]_{(-t)}$ per $d$ pari o dispari per far funzionare il Remark \ref{rem:Q_tau_J}} 
is trivial for $d$ odd ($d>1$) and is $(1+q)q^{n-1}
$ for $d$ even. For $d$ even a basis of the homology is given by the following generators
$$\frac{\DB (z_{n})}{1-t},  z_{n}.$$
\end{prop}	
\begin{proof}
This follows from  \cite[Prop.~3.2]{cal_sal_superelliptic} and by studying the long exact sequence associated to $$
0 \to \F[t^\pmu] \stackrel{[d]_{(-t)}}{\longrightarrow} \F[t^\pmu] \to \F[t]/([d]_{(-t)}) \to 0
$$ as in \cite[\S~4.2]{calmar}.
\end{proof}
%where for $e=1$ we have elements in $\cB'_2$, while for $e=2$ we have elements in $\cB''_2$.
%%
%% Sposto fino a qui
%%
%%	

\section{Exact sequences} \label{sec:exact_seq}

% \emph{[inizia Filippo coi teoremi principali]}

% \emph{-ricostruire la decomposizione data da Bianchi e arrivare alla successione di Mayer-Vietoris} \vskip 10pt

This section generalizes the results given in \cite[\S~4]{cal_sal_superelliptic}.

The 
%$d$-fold covering 
fibration $\pi: \confmd{n}^d \to \conf_n$ has a continuous section $s$  that can be defined as follows.
\begin{df} \label{def:section}
  If $p$ is a monic polynomial with $n$ distinct roots $x_1, \ldots, x_n$ such that $|x_i| <1,$ $i=1,\dots,n,$  we define
$$s: p \mapsto \left(p, z:= \frac{(\max_i |x_i|) +1}{2}, \sqrt[d]{p(z)}) \right).$$
Here, if  $p(z) = \prod_i(z-x_i)$ we have $\Re (z-x_i) >0, $ and we choose  
$\sqrt[d]{z-x_i}$ as the unique  $d$-th root with maximum real part; this defines $\sqrt[d]{p(z)} := \prod_i \sqrt[d]{z-x_i}$ as a continuous function.
\end{df}
The section $s:\conf_n \to \confmd{n}^d$ lifts the section  $\overline{s}:\conf_n \to \confm{n}$ taking $p \mapsto (p, z)$. It follows from the exact sequence of the pair $(\confmd{n}^d, \conf_n)$ the splitting
$H_i(\confmd{n}^d) \simeq H_i(\conf_n) \oplus H_i(\confmd{n}^d, \conf_n)$. In a similar way we have 
$H_i(\confm{n}) \simeq H_i(\conf_n) \oplus H_i(\confm{n}, \conf_n)$.	

\begin{rem}\label{rem:Cn_sections}
	The group $\Z/d$ acts on $\confmd{n}^d$ as a group of automorphisms of the covering $\confmd{n}^d \to \confm{n}$. Hence we can define $d$ sections  $s^{(j)}: \conf_n \to \confmd{n}^d$, for $j \in \Z/d$ as follows:
	$$s^{(j)}: p \mapsto \left(p, z:= \frac{(\max_i |x_i|) +1}{2}, e^{\frac{2\pi i j }{d}}\sqrt[d]{p(z)}) \right)$$
	where we choose $\sqrt[d]{p(z)}$ as above (and  $s=s^{(0)}$).
	
	We can include $S:\conf_n \times \Z/d \into \confmd{n}^d$ and the action of $\Z/d$ exchanges the components $\conf_n \times \{j\}$. Hence we can understand the inclusion $s:\conf_n \into \confmd{n}^d$ in homology via the following diagram:
	$$
	H_i(\conf_n) \stackrel{\Id \otimes 1}{\longrightarrow} H_i(\conf_n) \otimes \ring[t]/(1-(-t)^d) \stackrel{S_*}{\to} H_i(\confmd{n}^d) 
	$$
	and the composition is injective.
\end{rem}

\begin{rem} \label{rem:s_sprime}
	If $\gcd(n,d)=1$ any two sections $s^{(k)},s^{(j)}$ are homotopic. In fact, given $a,b$ integers such that $an+bd=1$ we can define in a unique way a continuous family of maps
	$
	s_t:\conf_n \to \confmd{n}
	$
	such that $s_0 = s^{(k)}$ and
	$$
	s_t: \mapsto \left(p, z_t:= e^{2\pi i t (k-j)a }\frac{(\max_i |x_i|) +1}{2}, \sqrt[d]{p(z_t)} \right)
	$$
	where we choose the $d$-th root $\sqrt[d]{p(z_t)}$ as above for $t=0$ and we extend it continuously for $t>0$. 
	Since $p(z_t)$ is a product of $n$ factors we have $s_1 = s^{(j)}$. As a consequence ${s^{(i)}}_* = {s^{(j)}}_*$.
\end{rem}

Generalizing the construction given in \cite[\S~4]{cal_sal_superelliptic} (see also \cite{bianchi}) for the case $d=2$, we consider the following decomposition.
We have $\totsp_n^d= \confmd{n}^d\cup N,$ where $N$ is a small tubular neighborhood of the subset $\confm{n-1}$ corresponding to the ramification locus. The intersection $\confmd{n}^d\cap N$ contracts onto  $\confm{n-1} \times S^1$. Moreover, a copy of the subspace $\conf_n = s(\conf_n)$ is contained into $\confmd{n}^d.$ 

Therefore the relative Mayer-Vietoris long exact sequence 
gives the following long exact sequence:
$$
\cdots \to H_i(\confm{n-1}) \oplus H_{i-1}(\confm{n-1}) \otimes H_1(S^1)\stackrel{\iota}{\to} H_i (\confm{n-1}) \oplus H_i(\confmd{n}^d, \conf_n) \to H_i(\totsp_n^d, \conf_n) \to \cdots
$$
From Kunneth decomposition the map $\iota$ factorizes taking the first (second) factor into the first (second) factor; so the exact sequence reduces to: 
\begin{equation}\label{eq:bia1}
\cdots \to  H_{i-1}(\confm{n-1}) \otimes H_1(S^1)\stackrel{\iota}{\to} H_i(\confmd{n}^d, \conf_n) \to H_i(\totsp_n^d, \conf_n) \to \cdots
\end{equation}

% \emph{-descrivere le identificazioni, anche in termini di inclusioni e quozienti di successioni spettrali} \vskip 10pt

The argument given in \cite[\S~4]{cal_sal_superelliptic} gives
$$
H_i(\Br_{n}; \st_n) = H_i(\conf_{n}; \st_n) \simeq H_i(\confm{n-1})=H_i(\Gb{n-1})
$$
where $\st_n$ is the permutation representation (see section 2) and $H_1(S^1 \times P) \simeq \st_n$. Moreover, from Remark \ref{rem:globalsection}:
$$
H_i (\totsp_n^d, \conf_n) \simeq H_{i-1}(\Br_n; H_1(\surf_n^d))
$$
and finally the term $ H_i(\confmd{n}^d, \conf_n)$ is isomorphic to $H_{i-1}(\Br_n; H_1(\ddiskP^d))$.

Therefore we can rewrite \eqref{eq:bia1} as an exact sequence involving non local homology of the braid groups:
\begin{equation}\label{eq:bia2}
\cdots\!\to\!\! H_{i-1}(\Br_n; H_1(S^1\! \times\! P)\!)\!\! \stackrel{\iota}{\to} \!\!H_{i-1}(\Br_n; H_1(\ddiskP^d)\!)\!\! \to\!\!  H_{i-1}(\Br_n; H_1(\surf_n^d)\!)\!\! \to\! \cdots
\end{equation}
Sequence \eqref{eq:bia2} is the long exact sequence for the homology of the group $\Br_n$ associated to the short exact sequence of coefficients
$$
0 \to H_1(S^1 \times P) \to H_1(\disk \times P) \oplus H_1(\ddiskP^d) \to H_1(\surf_n^d) \to 0
$$
coming from the Mayer-Vietoris long exact sequence associated to the decomposition: $\surf_n^d = \disk \times P \cup \ddiskP^d$, where  $\disk \times P \cap \ddiskP^d \simeq S^1 \times P$.

%\begin{df} \label{df:mu}
We recall from \cite[Def.~5]{cal_sal_superelliptic} 
the map 
$$
\mu: \confm{n-1} \times S^1 \ni (p,e^{it}) \mapsto (p_1 + \delta(\overline{p})e^{it}, \overline{p}) \in \confm{n}
$$
where $p= (p_1, \{p_2, \ldots, p_n\})$ in $\confm{n-1}$ and   $\overline{p} = \{p_1, \ldots, p_n\},$ while $$\delta(\overline{p}):= \frac{1}{2}\min \left( \{ |p_1 - p_i|,  2 \leq i  \leq n\} \cup \{ 1-|p_1|\}
\right).$$

%\end{df}

We have a commutative diagram
\begin{equation}\label{diag:trerighe0}
\begin{tabular}{c}
\xymatrix @R=2pc @C=2pc {
H_{i-1}(\Br_n; H_1(S^1 \times P)) \ar[r]^\iota \ar[d]^\simeq & H_{i-1}(\Br_n; H_1(\ddiskP^d))\ar[d]^\simeq  \\
H_{i-1}(\confm{n-1}) \otimes H_1(S^1) \ar[d]^{\mu_*} \ar[r] & H_i(\confmd{n}^d, \conf_n) \\
H_i(\confm{n}) 
%\ar[r]^{\tau} 
&H_i(\confmd{n}^d)  \ar[u]_J
}
\end{tabular}
\end{equation}
where $\mu_*$ is induced by the map $\mu$ above 
and $J$ is induced by the inclusion $$\confmd{n}^d \to (\confmd{n}^d, \conf_n).$$
Some properties of  $\mu_*$ are investigated in \cite[Prop.~ 4.1,  Prop.~4.2, Rmk.~5]{cal_sal_superelliptic}. In particular we have that  $\mu_*$ is injective and when we consider homology with coefficients in a field $\F$ of characteristic $0$ the map
$$
\mu_*: H_{i-1}(\confm{n-1}) \otimes H_1(S^1) \to H_i(\confm{n})
$$
is an isomorphism.

%\emph{Leggendo Bianchi si capisce che $\mu$ \`e definita nel teorema 12 e poi nella sezione 3. Basterebbe quindi mostrare che la decomposizione che si trova dal teorema 12 corrisponde a quella della ss data dalla filtrazione del complesso.}
Recall the notation $[d]_q$ for the $q$-analog $1+ q + \cdots + q^{d-1}$.

\begin{prop}\label{prop:tau1}
Let $\tau: H_i(\confm{n}) \to H_i(\confmd{n}^d)$ be the transfer map induced by the $d$-fold covering $\confmd{n} ^d\to \confm{n}$.
The following diagram commutes:
\begin{equation}\label{diag:trerighe}
\begin{tabular}{c}
\xymatrix @R=2pc @C=2pc {
	H_{i-1}(\Br_n; H_1(S^1 \times P)) \ar[r]^\iota \ar[d]^\simeq & H_{i-1}(\Br_n; H_1(\ddiskP^d))\ar[d]^\simeq  \\
	H_{i-1}(\confm{n-1}) \otimes H_1(S^1) \ar[d]^{\mu_*} \ar[r] & H_i(\confmd{n}^d, \conf_n) \\
	H_i(\confm{n}) \ar[r]^{\tau} &H_i(\confmd{n}^d)  \ar[u]_J
}
\end{tabular}
\end{equation}
\end{prop}

\begin{prop}\label{prop:commut_tau}
The following diagram commutes:
$$
\begin{tabular}{c}
\xymatrix @R=2pc @C=2pc {
H_i(\confm{n})\ar[d]^\simeq \ar[r]^{\tau} &H_i(\confmd{n}^d) \ar[d]^\simeq \\
H_i(\Gb{n}; \ring[t]/(1+t)) \ar[r]^(0.45){[d]_{(-t)}} & H_i(\Gb{n}; \ring[t]/(1-(-t)^d))	.
}
\end{tabular}
$$
where in the bottom row we are considering the map induced by the $[d]_{(-t)}$-multiplication map $C_*(\Gb{n}, \ring[t]/(1+t)) \to C_*(\Gb{n}, \ring[t]/(1-(-t)^d)).$ 
\end{prop}
The proofs of Propositions \ref{prop:tau1}, \ref{prop:commut_tau} are analogous to the proofs of  \cite[Prop.~4.3, Prop.~4.4]{cal_sal_superelliptic}.
%\section{title}
%\vskip 10pt \emph{-interpretazioni in termini di calcoli di Callegaro-Marin}\vskip 10pt

We also have the following analogue of \cite[Rmk.~6]{cal_sal_superelliptic}:
\begin{rem} \label{rem:J}
%\todo{AGGIORNATO}
Consider the isomorphism $H_i(\confmd{n}^d) \simeq H_i(\Gb{n}; \ring[t]/(1-(-t)^d))$. Let $\ring$ be a field of characteristic $p$. For $p \nmid d$ the second term decomposes as $$H_i(\Gb{n}; \ring[t]/(1+t)) \oplus H_i(\Gb{n}; \ring[t]/[d]_{(-t)})$$ and moreover 
%\todo{controlla come cambia la condizione}
{for $n$ odd} and $p \nmid d$    
the term $H_i(\Gb{n}; \ring[t]/[d]_{(-t)})$ is trivial, which follows from the fact that for $n$ odd the module $H_i(\Gb{n}; \ring[t^\pmu])$ has $(1+t)$-torsion and from the homology long exact sequence associated to
	$$
	0 \to \ring[t^\pmu] \stackrel{[d]_{(-t)}}{\longrightarrow} \ring[t^\pmu ] \to \ring[t]/[d]_{(-t)} \to 0
	$$
since $[d]_{(-t)} \equiv d \mod (1+t)$.
Since for $n$ even  the module $H_i(\Gb{n}; \ring[t^\pmu])$ has $(1-t^2)^k$-torsion for suitable $k$ (see \cite[Thm.~4.12]{calmar}), the same argument applies for $n$ even, when $d$ is odd and $p\nmid d$. In fact for $p \neq 2$ we have that $[d]_{(-t)}$ is co-prime both with $(1+t)$ and with $(1-t)$ and, for $p$ odd, $(1+t)$ and $(1-t)$ are co-primes; for $p=2$ and $d$ odd we have that $[d]_{(-t)}$ and $(1-t^2)$ are co-primes. 

	%to\cite[Sec.~4.5]{calmar}). 
	In particular, under the condition stated above the homology groups $H_i(\confmd{n}^d) \simeq H_i(\Gb{n}; \ring[t]/(1-(-t)^d))$ and $H_i(\confm{n}) \simeq H_i(\Gb{n}; \ring[t]/(1+t))$ are isomorphic and the isomorphism is induced by the quotient map $\ring[t]/(1-(-t)^d) \to \ring[t]/(1+t)$ (again, here we are using that $p\nmid d$ and $n$ is odd). %In fact it turns out that this map is the multiplication by $d$, hence 
	So we can consider the commuting diagram
	$$
	\begin{tabular}{c}
	\xymatrix @R=2pc @C=2pc {
		0 \ar[r]& H_i(\conf_n) \ar[d]^\simeq \ar[r] & H_i(\confmd{n}^d)	\ar[d]^\simeq \ar[r]^J & H_i(\confmd{n}^d, \conf_n) \ar[d]^\simeq  \ar[r] & 0\\
		0 \ar[r] & H_i(\conf_n) \ar[r] & H_i(\confm{n}) \ar[r]^{\overline{J}} & H_i(\confm{n}, \conf_n) \ar[r] & 0.
	}
	% H_i(\confm{n})
	% H_i(\confmd{n})
	% H_i(\Gb{n}; \ring[t]/(1+t))
	% H_i(\Gb{n}; \ring[t]/(1-t^2))
	\end{tabular}
	$$
	where the last vertical map is an isomorphism from the five lemma. 
	From the right square we have that the map $J$ corresponds to the homomorphism
	$$
\overline{J}:	H_i(\Gb{n}; \ring[t]/(1+t)) \to H_i((\Gb{n}, \Ga{n-1}); \ring[t]/(1+t))
	$$
	associated to the inclusion $\Ga{n-1} \into \Gb{n}$ induced by $\conf_n \into \confm{n}$.  From the short exaxt sequence in the second row of the diagram above we have that the homomorphism $$\overline{J}: H_i(\confm{n}) = H_i(\confm{n}, \conf_n) \oplus H_i(\conf_n) \to H_i(\confm{n}, \conf_n)$$ is the projection on the first term of the direct sum.
\end{rem}

\begin{lem}\label{lem:tau_invertible}
If $p$ is a prime such that $p \nmid d$ or $p=0$, $\ring$ is a field of characteristic $p$ then the homomorphism
$$
\overline{\tau}:H_i(\Gb{n}; \ring[t]/(1+t)) \stackrel{[d]_{(-t)}}{\longrightarrow} H_i(\Gb{n}; \ring[t]/(1-(-t)^d))
$$
is invertible if at least one between $n$  and $d$ is odd. 
%\todo{vale anche per $n$ even, $d$ odd, $p$ odd - AGGIORNATO}.
\end{lem}
\begin{proof}
This follows since, for odd $n$, the homology group $H_i(\Gb{n}; \ring[t^\pmu])$ has $(1+t)$-torsion, while for $n$ even the group has torsion of order $(1-t^2)^k$ for suitable $k.$
%, \cite[Sec.~4.5]{calmar}

%If $p \neq 2$ the ideals $(1+t)$ and $(1-t)$ are coprimes in $\R[t^\pmu]$.

%Moreover $(1-(-t)^d)= (1+t)[d]_{(-t)}$ and $(1+t, [d]_{(-t)}) = (1+t,d)$ and hence, since $p \nmid d$, $[d]_{(-t)}$ is invertible mod $(1+t)$. 
%On the other hand $(1-t, [d]_{(-t)}) = (1)$ for $d$ odd.

For 
%$p \neq 2$ and 
$p \nmid d$  (and with $d$ odd when $n$ is even) we have that  $$H_i(\Gb{n}; \ring[t]/[d]_{(-t)}) = 0$$ (see Remark \ref{rem:J}) and hence 
$$
H_i(\Gb{n}; \ring[t]/(1-(-t)^d)) \simeq H_i(\Gb{n}; \ring[t]/(1+t))
$$
and the map $\overline{\tau}$ is equivalent to the multiplication map
$$
H_i(\Gb{n}; \ring[t]/(1+t)) \stackrel{[d]_{(-t)}}{\longrightarrow} H_i(\Gb{n}; \ring[t]/(1+t))
$$
and $[d]_{(-t)}$ is invertible in $\ring[t]/(1+t)$.
\end{proof}
\begin{rem} \label{rem:Q_tau_J}
In characteristic $p=0$, for $n$ even the decompositions 
\begin{align*}
H_i(\confmd{n}^d) \simeq & H_i(\Gb{n}; \ring[t]/(1-(-t)^d)) = \\ = & H_i(\Gb{n}; \ring[t]/(1+t)) \oplus H_i(\Gb{n}; \ring[t]/[d]_{(-t)})
\end{align*}
and $$H_i(\confm{n}) = H_i(\confm{n}, \conf_n) \oplus H_i(\conf_n)$$ give the following consequences.
%for $n$ odd and $\F$ a field of characteristic $0$. 
Since 
%$H_i(\Gb{n}; \ring[t]/(1-t))$ has Poincar\'e polynomial $(1+q)q^{n-1}$ (Proposition \ref{prop:hrazionale_minus_t})
$H_i(\Gb{n}; \ring[t]/([d]_{(-t)}))$ is trivial for $d$ odd and has Poincar\'e polynomial $(1+q)q^{n-1}$ for $d$ even (see Proposition \ref{prop:hrazionale_minus_t})
and since $H_*(\conf_n)$ has Poincar\`e polynomial $(1+q)$ (see \cite[\S~3.1]{cal_sal_superelliptic}), we have that the map
$$
J: H_i(\confmd{n}^d)  \to H_i(\confmd{n}^d, \conf_n) 
$$
is an isomorphism for $i>1$.
Moreover the argument of Lemma \ref{lem:tau_invertible} implies that
the map
$$\tau: H_*(\confm{n}) \to H_*(\confmd{n}^d)$$ is injective and its cokernel is trivial for $d$ odd and has Poincar\'e polynomial $(1+q)q^{n-1}$ for $d$ even.
\end{rem}

\begin{prop} \label{prop:coker_n_pari}
Let $n$ be even and $\F$ a field of characteristic $0$. Then for $i>1$ the map 
$$
\iota: H_{i-1}(\Br_n; H_1(S^1 \times P))  \to H_{i-1}(\Br_n; H_1(\ddiskP^d))
$$
is injective; its cokernel has rank $1$ if $d$ is odd and $i = n-1, n-2$; otherwise its rank is $0$.
\end{prop}
\begin{proof}
The proposition follows from \cite[Rmk.~5]{cal_sal_superelliptic}
%\ref{rem:Q_mu} 
and from Remark \ref{rem:Q_tau_J}.
\end{proof}
\begin{thm}\label{teo:tau_restricted}
Consider the inclusions $H_i(\conf_n) \subset H_i(\confm{n}, \conf_n)$ 
%\oplus H_i(\confm{n})$ 
and $ H_i(\conf_n) \subset H_i(\confmd{n}^d)$ % = H_i(\confmd{n}^d, \conf_n) \oplus$
associated to the sections $\overline{s}:\conf_n \into \confm{n}$  and $s:\conf_n \into \confmd{n}^d$.
If $n$ is odd 	
%\todo{per $d$ dispari questo vale anche per $n$ pari, ma solo in caratteristica dispari che non divida $d$ - AGGIORNATO}
%$(n,d)=1$
%the following inclusion holds: 
$$\tau(H_*(\conf_n)) \subset H_*(\conf_n)$$
and for $x \in H_*(\conf_n)$ we have %that 
$\tau(x) = d x$.
Moreover if $n$ is even, $d$ is odd and $\ring$ is a ring of characteristic $p$ with $p \nmid d$, then for $x \in H_*(\conf_n; \ring)$ we have 
$\tau(x) = d x$.
\end{thm}
\begin{proof}
%\todo{completare usando argomento di ss su $\Z$...}
%In this proof we assume $\ring$ to be any ring of coefficients, with trivial action. Where the ring of coefficient is not explicit, we will understand homology with coefficient in $R$.

The complex $C_*(\Ga{n-1}), \ring)$ described in Section \ref{sec:homol_artin}, that computes the homology $H_*(\conf_n; \R)$, can be seen as a subcomplex of $$C_*(\Gb{n}), \ring) \simeq C_*(\Gb{n}), \ring[t]/(1+t))$$ mapping $A \mapsto \overline{0}A$. This inclusion of complexes induces the homology homomorphism associated to the section $\overline{s}:\conf_n \into \confm{n}$. Moreover, this map of complexes lift to a map
$\widetilde{s}_*: C_*(\Ga{n-1}), \ring) \to C_*(\Gb{n}), \ring[t^\pmu ])$ and hence to the map $\widetilde{s}_*: C_*(\Ga{n-1}), \ring) \to C_*(\Gb{n}), \ring[t]/(1-(-t)^d))$ that induces the homology homomorphism associated to the section $s:\conf_n \into \confmd{n}^d$.

This implies that we have the following commutative diagram with exact columns:
\begin{equation*}
\begin{tabular}{c}
\xymatrix @R=2pc @C=2pc {
& H_*(\Gb{n}), \ring[t^\pmu ])\ar[r]^1 \ar[d]^{1+t}& H_*(\Gb{n}), \ring[t^\pmu ])\ar[d]^{1-(-t)^d}\\
&H_*(\Gb{n}), \ring[t^\pmu ])\ar[r]^{[d]_{(-t)}} \ar[d]& H_*(\Gb{n}), \ring[t^\pmu ])\ar[d]\\
& H_*(\Gb{n}), \ring[t ]/(1+t))\ar[r]^(0.47)\tau& H_*(\Gb{n}), \ring[t ]/(1-(-t)^d)) \\
H_*(\conf_n; \ring)\ar@{-->}[r] \ar@/^2pc/[uur]^{\widetilde{s}_*} \ar[ur]^(0.4){\overline{s}_*}& H_*(\conf_n; \ring)\ar[ur]^(0.4){s_*}\ar@/^1pc/[uur]|(0.37){\hole\hole}^(0.6){\widetilde{s}_*}& 
}
\end{tabular}
\end{equation*}
Notice that the restriction of the map $\tau$ to the image of $\overline{s}_*$ lifts, via $\widetilde{s}_*$, to the multiplication by $[d]_{(-t)}$. 

We claim that for $n$ odd the image of $s_*$ in $H_*(\Gb{n}), \ring[t ]/(1-(-t)^d))$ has only $(1+t)$-torsion. This follows since any the inclusion 
$$
C_*(\Ga{n-2}), \ring) \to C_*(\Ga{n-1}), \ring)
$$
defined on generators by
$$
A \mapsto 0A
$$
induces an isomorphism in homology
and hence any class $x \in s_*(H_*(\conf_n; \ring))$ is represented by  a cycle of the form $\overline{0}0x'$ and we have the relation $(1+t)  \overline{0}0x' = \DB \overline{1}0x'$.

For $n$ even and $d$ odd, when $\ring$ is a field of characteristic $p$ such that $p \nmid d$ we have the $R[t]$-module $H_*(\Gb{n}), \ring[t ]/(1-(-t)^d))$ has only $(1+t)$-torsion and is isomorphic to $H_*(\Gb{n}), \ring[t ]/(1+t))$, as seen in Remark \ref{rem:J}, where the isomorphism is induced by the quotient map $\ring[t ]/(1-(-t)^d) \mapsto  \ring[t ]/(1+t)$.

In both cases this imply that we have the inclusion $\tau(H_*(\conf_n)) \subset H_*(\conf_n)$ and, since $[d]_{(-t)} \equiv d \mod (1+t)$, $\tau(x) = d x$ for $x \in H_*(\conf_n)$.

\end{proof}

From theorem \ref{teo:tau_restricted}
and Lemma \ref{lem:tau_invertible} we obtain:
\begin{cor} \label{cor:tau_restricted}
%		\todo{per $d$ dispari questo vale anche per $n$ pari, ma solo con $R$ di caratteristica dispari che non divide $d$ - AGGIORNATO}
If  $\ring$ is a field of characteristic $p$, $p \nmid d$, and at least one between $d$ and $n$ is odd,
% and $(n,d)=1$, 
then the projection on $H_i(\confmd{n}^d,\conf_n)$ of the restriction of the homomorphism $\tau$ to $H_i(\confm{n},\conf_n)$
$$
	\tau_|:H_i(\confm{n},\conf_n) \to H_i(\confmd{n}^d,\conf_n)
	$$
	is an isomorphism.
\end{cor}
\begin{proof}
Under the given assumptions $\tau$ is an isomorphism (Lemma \ref{lem:tau_invertible}) and its restriction to $H_i(\conf_n)$ maps to $H_i(\conf_n)$. So the corollary follows.
\end{proof}

We now describe the odd torsion of the homology $H_i(\Br_n; H_1(\surf_n^d)),$ by using  the map 
%$\mu_*$ and 
$J$ in the diagram \eqref{diag:trerighe}.

Let $\Filt$ and $E^r_{ij}$ be the filtration and the associated spectral sequence  defined in \cite[\S~4]{cal_sal_superelliptic}. The proof of the following lemma is completely analogous to the one of \cite[Lem.~4.9]{cal_sal_superelliptic}.
\begin{lem}\label{lem:ss_incl}
%nella ss la mappa di inclusione è l'iso con la prima colonna
Let $E^2_{ij} =  H_j(\Br(n-i);\ring[t]/([d]_{(-t)}) \Rightarrow H_i(\Gb{n}; \ring[t]/([d]_{(-t)}))$ be the spectral sequence induced by the
filtration $\Filt$. The inclusion $\Ga{n-1} \into \Gb{n}$ induces the isomorphism %of the term 
$$H_j(\Ga{n-1}; \ring[t]/([d]_{(-t)})) 
\simeq
%= H_j(\Br(n);\ring[t]/([d]_{(-t)})) $ 
%with the submodule $
E^2_{0j}$$ for all $j$.
\end{lem}

\begin{thm}\label{th:only_d_torsion}
%		\todo{per $d$ dispari questo vale anche per $n$ pari? - AGGIORNATO}
%Let $n$ be an odd integer and let $g = (n-1)/2$ and let $\sym{g}= 
Let at least one between $n$ and $d$ be odd.
%Moreover assume that $\gcd(n,d)=1$.
Consider the integral monodromy representation of the braid group $\Br_n$ on the group $H_1(\surf_n^d;\Z)$. 
Then the homology  $H_i(\Br_n; H_1(\surf_n^d;\Z))$ is a torsion $\Z$-module
with only $p^j$ torsion for $p \mid d$.
% and the torsion of its elements is a power of $2$.
%and for any odd prime $p$ 
%the homology $H_i(\Br_n; \sym{g})$
%has no $p$-torsion.	
\end{thm}
\begin{proof}
Let  $n$ and $d$ be as in the theorem, and let $p \nmid d$.  
It derives from the description of the map $\mu_*$ (see \cite[Prop.~4.1 and Cor.~4.10]{cal_sal_superelliptic}), from the results about the map $\tau$ (Proposition \ref{prop:tau1}, Lemma \ref{lem:tau_invertible}, Corollary \ref{cor:tau_restricted}) and from 
Remark \ref{rem:J} concerning the map $J$ that
the map $\iota$ in diagram \eqref{diag:trerighe} is an isomorphism 
in characteristic $p.$ %\todo{ok?}
Then the result follows from sequence \eqref{eq:bia2}.
\end{proof}

\section{A first bound for torsion order} \label{s:4tor}
%\todo{adattare questa sezione e la successiva al caso ge\-ne\-rale. come?}
%\todo{ma un esempio della rappresentazione? come agiscono i generatori del gruppo di trecce?}
In this section we prove that if \ $p^i \mid d$ and \ $p^{i+1} \nmid d$, then torsion in $H_{i}(\Br_n; H_1(\surf_n^d))$ appears with order at most $p^{i+1}$. 

Following lemma generalizes \cite[Lem.~5.2]{cal_sal_superelliptic}.
\begin{lem} \label{lem:no2tor2}
%		\todo{per $d$ dispari questo vale anche per $n$ pari? magari s\`i, ma il calcolo del Bockstein in generale sembra duro.}
Let $R = \Z$. For $n$ odd, the homology $$H_i(\confmd{n}^d) \simeq H_i(\Gb{n}; \ring[t]/(1-(-t)^d))$$ has no $p^2$-torsion for any prime $p$. %\todo{adattare la dimostrazione, che segue pari pari\\  $\beta_p x_0 = 0,$\\ $\beta_p x_i = y_i$ se $i>0$,\\$\beta_p y_i = 0$.}
\end{lem}
\begin{proof} 
It will suffice to show that the dimension over $\F_p$ of the homology of the complex $(H_*(\Gb{n}; \F_p[t]/(1-(-t)^d)), \beta_p)$, where $\beta_p$ is  the Bockstein homomorphism, is the same as the dimension over $\Q$ of $H_*(\Gb{n}; \Q[t]/(1-(-t)^d))$ (see \cite[Thm.~3E.4]{hatcher_at}).
We split our proof in several steps
\begin{enumerate}[label= \emph{Step \alph*)}, wide, labelindent=0pt]
	\item  According to \cite[Thm.~6.1, case 
%$r=2$ and 
$n$ odd]{lehr04}, the Poincar\'e polynomial of the homology groups
$H_*(\Gb{n}; \Q[t]/(1-(-t)^d)) 
%= H_*(B(4,2,n); \Q)
$
is
$$
P(\Gb{n},t)= 
%\left\{
%\begin{array}{cl}
(1+t)(1+t+t^2+ \cdots + t^{n-1}). 
%& 
%\mbox{ if } n \mbox{ is odd,} \\
%(1+t)(1+t+t^2+ \cdots + t^{n-1}) + (t^{n-1} +t^n) & \mbox{ otherwise.}
%\end{array}
%\right. 
$$

\item
The explicit computation of the Bockstein homomorphism $\beta_p$ of the homology group $H_*(\Gb{n}; \F_p[t]/(1-(-t)^d))$ is the following.
%\todo{dividere $p=2$ e $p$ dispari}
Let $$0 \to \Z_p[t]/(1-(-t)^d) \stackrel{i_p}{\longrightarrow} \Z_{p^2}[t]/(1-(-t)^d) \stackrel{\pi_p}{\longrightarrow} \Z_p[t]/(1-(-t)^d) \to 0$$ be the short exact sequence of coefficients. Then for $p=2$ (see Prop.~\ref{prop:generators})
$$\frac{(1-(-t)^d)}{1+t} z_{c+1} x_{i_1} \cdots x_{i_k} = \pi_2 (\frac{(1-(-t)^d)}{1+t} z_{c+1} x_{i_1} \cdots x_{i_k} )$$
and
$$
\DB \frac{(1-(-t)^d)}{1+t} z_{c+1} x_{i_1} \cdots x_{i_k} = \sum_{
	\scsc{j=1, \ldots, k}{i_j>1}
}
2 \frac{(1-(-t)^d)}{1+t} z_{c+1} x_{i_1} \cdots x_{i_j-1}^2 \cdots x_{i_k} =
$$
$$
= i_2 \left(\sum_{
	\scsc{j=1, \ldots, k}{i_j>1}
}
\frac{(1-(-t)^d)}{1+t} z_{c+1} x_{i_1} \cdots x_{i_j-1}^2 \cdots x_{i_k} \right)
$$ and hence, using the notation introduced in Proposition \ref{prop:generators} we have
\begin{equation}\label{bockstein1}
\beta_2 \gp (z_c, x_0x_{i_1} \cdots x_{i_k})   =\sum_{
	\scsc{j=1, \ldots, k}{i_j>1}
} \gp (z_c, x_0 x_{i_1} \cdots x_{i_j-1}^2 \cdots x_{i_k}) 
\end{equation}
for all generators of the form given in \eqref{generators1}.
Moreover 
$$\frac{\DB(z_{c+1}x_{i_1} \cdots x_{i_k})}{(1+t)} =
$$
$$
= \pi_2 \left( \frac{1}{(1+t)}\left( \DB(z_{c+1}x_{i_1} \cdots x_{i_k}) - 2\sum_{
	\scsc{j=1, \ldots, k}{i_j>1}
}
z_{c+1} x_{i_1} \cdots x_{i_j-1}^2 \cdots x_{i_k} ) \right) \right) $$
and
$$
\DB \left( \frac{1}{(1+t)}\left( \DB(z_{c+1}x_{i_1} \cdots x_{i_k}) - 2\sum_{
		\scsc{j=1, \ldots, k}{i_j>1}
	}
	z_{c+1} x_{i_1} \cdots x_{i_j-1}^2 \cdots x_{i_k} ) \right) \right) =
$$
$$
=\DB \left( \frac{1}{(1+t)}\left(- 2\sum_{
	\scsc{j=1, \ldots, k}{i_j>1}
}
z_{c+1} x_{i_1} \cdots x_{i_j-1}^2 \cdots x_{i_k} ) \right) \right) =
$$
$$
= -2 \sum_{
	\scsc{j=1, \ldots, k}{i_j>1}
} \frac{\DB(z_{c+1} x_{i_1} \cdots x_{i_j-1}^2 \cdots x_{i_k} ) }{1+t}=
i_2 \left( \sum_{
	\scsc{j=1, \ldots, k}{i_j>1}
} \frac{\DB(z_{c+1} x_{i_1} \cdots x_{i_j-1}^2 \cdots x_{i_k} ) }{1+t} \right)
$$
hence we have
\begin{equation}\label{bockstein2}
\beta_2 \gs (z_c, x_0x_{i_1} \cdots x_{i_k}) =\sum_{
	\scsc{j=1, \ldots, k}{i_j>1}
}  \gs (z_c, x_{0} x_{i_1}  \cdots x_{i_j-1}^2 \cdots  x_{i_k})
\end{equation}
for all generators of the form given in \eqref{generators2}.

For $p$ odd we have the following analogous computation (again, see Prop.~\ref{prop:generators})
$$\frac{1-(-t)^d}{1+t} 
z_{c+1} h^{r-1} y_{j_1} \cdots y_{j_l }x_{i_1} \cdots x_{i_k}
= \pi_p \left(\frac{1-(-t)^d}{1+t} 
z_{c+1} h^{r-1} y_{j_1} \cdots y_{j_l }x_{i_1} \cdots x_{i_k}
\right)$$
and
$$
\DB \frac{(1-(-t)^d)}{1+t} z_{c+1}h^{r-1} y_{j_1} \cdots y_{j_l } x_{i_1} \cdots x_{i_k} = $$ $$ = \sum_{
	\scsc
	{j=1, \ldots, k}{i_j\geq 1}
}
p \frac{(1-(-t)^d)}{1+t} z_{c+1}h^{r-1} y_{j_1} \cdots y_{j_l } x_{i_1} \cdots x_{i_{j-1}} y_{i_j} x_{i_{j+1}} \cdots x_{i_k} =
$$
$$
= i_p \left(\sum_{
	\scsc
	{j=1, \ldots, k}
	{i_j\geq 1}
}
\frac{(1-(-t)^d)}{1+t} z_{c+1} h^{r-1} y_{j_1} \cdots y_{j_l } x_{i_1} \cdots x_{i_{j-1}} y_{i_j} x_{i_{j+1}} \cdots x_{i_k} \right)
$$ and hence, using the notation introduced in Proposition \ref{prop:generators} we have
\begin{equation}\label{bockstein1p}
\beta_p \gp (z_c, h^{r} y_{j_1} \!\cdots\! y_{j_l }x_{i_1}\! \cdots\! x_{i_k})   =\!\!\! \!\sum_{
\scsc
{j=1, \ldots, k}{i_j\geq 1}
} \! \!\!\gp (z_c, h^{r} y_{j_1}\! \cdots \!y_{j_l } x_{i_1} \!\cdots\! x_{i_{j-1}} y_{i_j} x_{i_{j+1}} \!\cdots\! x_{i_k}) 
\end{equation}
for all generators of the form given in \eqref{generators1p}.
Moreover 
$$\frac{\DB(z_{c+1}h^{r-1} y_{j_1} \cdots y_{j_l }x_{i_1} \cdots x_{i_k})}{(1+t)} = $$ $$=\pi_p \left( \frac{1}{(1+t)}\left( \DB(z_{c+1}h^{r-1} y_{j_1} \cdots y_{j_l }x_{i_1} \cdots x_{i_k}) \right) \right) + $$ $$ - p\pi_p \left( \frac{1}{(1+t)}\left(\sum_{
	\scsc
	{j=1, \ldots, k}{i_j\geq 1}
}
z_{c+1} h^{r-1} y_{j_1} \cdots y_{j_l }x_{i_1} \cdots x_{i_{j-1}} y_{i_j} x_{i_{j+1}} \cdots x_{i_k} ) \right) \right) $$
and computing $\DB$ on the term above we get
%$$
%\DB \left( \frac{1}{(1+t)}\left( \DB(z_{c+1}x_{i_1} \cdots x_{i_k}) - 2\sum_{
%	\scsc{j=1, \ldots, k}{i_j>1}
%}
%z_{c+1} x_{i_1} \cdots x_{i_j-1}^2 \cdots x_{i_k} ) \right) \right) =
%$$
%$$
%=\DB \left( \frac{1}{(1+t)}\left(- 2\sum_{
%	\scsc{j=1, \ldots, k}{i_j>1}
%}
%z_{c+1} x_{i_1} \cdots x_{i_j-1}^2 \cdots x_{i_k} ) \right) \right) =
%$$
$$
 -p \sum_{
	\scsc
	{j=1, \ldots, k}{i_j\geq 1}
} \frac{\DB(z_{c+1}h^{r-1} y_{j_1} \cdots y_{j_l } x_{i_1} \cdots x_{i_{j-1}} y_{i_j} x_{i_{j+1}} \cdots x_{i_k} ) }{1+t}= $$ $$ =
i_p \left(- \sum_{
	\scsc
	{j=1, \ldots, k}{i_j\geq 1}
} \frac{\DB(z_{c+1} h^{r-1} y_{j_1} \cdots y_{j_l }x_{i_1} \cdots x_{i_{j-1}} y_{i_j} x_{i_{j+1}} \cdots x_{i_k} ) }{1+t} \right)
$$
hence up to change of sign we have
\begin{equation}\label{bockstein2p}
\beta_p\! \gs (z_c, h^{r} y_{j_1} \!\cdots\! y_{j_l }x_{i_1} \!\cdots\! x_{i_k}) =\!\!\!\!\!\sum_{
	\scsc
	{j=1, \ldots, k}{i_j\geq 1}
}\!\!\!  \gs (z_c, h^{r} y_{j_1} \!\cdots\! y_{j_l } x_{i_1} \! \cdots\! x_{i_{j-1}} y_{i_j} x_{i_{j+1}} \!\cdots \! x_{i_k})
\end{equation}
for all generators of the form given in \eqref{generators2p}.

%\todo{completare il calcolo facendo l'analogo for $p$ odd...}
\item Next we need to consider the following definition.
\begin{df} \label{def:modules}%\todo{adattare a $p$ dispari}
Let $a,b$ be two non-negative integers, with  $a \in \N_{>0}$, $b \in \{0,1\}.$ 
Moreover let $I = (i_1, \ldots, i_k)$, $J = (j_1, \ldots, j_h)$, where we assume that:
\begin{enumerate}[label=(\roman*)]
	\item $j_1 < \cdots <j_h$,
%	\item[ii)] $\min I \geq 1$,
	\item $\min J \geq 2$,
	\item for all $s \in 1, \ldots, k$ there exists an integer $t \in 1, \ldots, h$ such that $i_s+1 = j_t$ 
\end{enumerate}  
We define the following sub-modules of $H_*(\Gb{n}; \F_2[t]/(1-(-t)^d))$:
%\begin{enumerate}
%\item[i)] 
$$\MM{c,a,b,I,J}_2 := \langle
\gs (z_c, x_0^a x_1^b x_{i_1}^2 \cdots x_{i_k}^2\epsilon(x_{j_1})\cdots \epsilon(x_{j_h})) | \mbox{ where }\epsilon(x_{j_t}) = x_{j_t} \mbox{ or }  x_{j_t-1}^2 \rangle.$$
and
$$\MMp{c,a,b,I,J}_2 := \langle
\gp (z_c, x_0^a x_1^b x_{i_1}^2 \cdots x_{i_k}^2\epsilon(x_{j_1})\cdots \epsilon(x_{j_h})) | \mbox{ where }\epsilon(x_{j_t}) = x_{j_t} \mbox{ or }  x_{j_t-1}^2 \rangle.$$

Moreover let  $p$ be an odd prime and  %let $a,b$ be two non-negative integers, with  $a \in \N_{>0}$, $b \in \{0,1\}.$ 
let $I = (i_1, \ldots, i_k)$, $J = (j_1, \ldots, j_h)$, where we assume that:
\begin{enumerate}[label=(\roman*')]
	\item $j_1 < \cdots <j_h$,
	%	\item[ii)] $\min I \geq 1$,
	\item $\min J \geq 1$,
	\item for all $s \in 1, \ldots, k$ there exists an integer $t \in 1, \ldots, h$ such that $i_s = j_t$ 
\end{enumerate}  
We define the following sub-modules of $H_*(\Gb{n}; \F_p[t]/(1-(-t)^d))$:
%\begin{enumerate}
%\item[i)] 
$$\MM{c,a,b,I,J}_p := \langle
\gs (z_c, h^a x_0^b y_{i_1} \cdots y_{i_k}\epsilon(x_{j_1})\cdots \epsilon(x_{j_h})) | \mbox{ where }\epsilon(x_{j_t}) = x_{j_t} \mbox{ or }  y_{j_t} \rangle.$$
and
$$\MMp{c,a,b,I,J}_p := \langle
\gp (z_c, h^a x_0^b y_{i_1} \cdots y_{i_k}\epsilon(x_{j_1})\cdots \epsilon(x_{j_h})) | \mbox{ where }\epsilon(x_{j_t}) = x_{j_t} \mbox{ or }  y_{j_t-1} \rangle.$$
%\end{enumerate}
\end{df}
\begin{rem} \label{rem:Jvuoto}
%\todo{adattare a $p$ dispari?}
Notice that conditions (iii) and (iii') above imply that if $J= \emptyset$ then also $I = \emptyset$ and hence for any prime $p$ the modules $\MM{c,a,b,\emptyset,\emptyset}_p$ and $\MMp{c,a,b,\emptyset,\emptyset}_p$ have rank $1$ concentrated in degree $c+b$ and $c+b+1$ respectively. 
\end{rem}

The modules $\MM{c,a,b,I,J}_p$ and $\MMp{c,a,b,I,J}_p$ are  free $\Z_p[t]/(1+t)$-modules, closed for $\beta_p$, as follows from formulas \eqref{bockstein1}, \eqref{bockstein2}, \eqref{bockstein1p}, \eqref{bockstein2p}.
\item If $J \neq \emptyset$, the complexes $(\MM{c,a,b,I,J}_p, \beta_p)$ and $(\MMp{c,a,b,I,J}_p, \beta_p)$ are acyclic. This fact can be proven by the same argument used in \cite[Lem.~4.4]{cal06}.
The argument can be expressed with the following statement:
\begin{lem} \label{lem:boolean}
%	\todo{adattare a $p$ dispari}
Let $\bool{J}$ be the chain complex with $\F_p$ coefficients associated to the boolean lattice of the subsets of $J$. The complexes $(\MM{c,a,b,I,J}_p, \beta_p)$ and $(\MMp{c,a,b,I,J}_p, \beta_p)$ are isomorphic to $\bool{J}$. In particular if $J \neq \emptyset$ the complexes $(\MM{c,a,b,I,J}_p, \beta_p)$ and $(\MMp{c,a,b,I,J}_p, \beta_p)$ are acyclic.
\end{lem}
\begin{proof}[Proof of Lemma~{\rm\ref{lem:boolean}}]
Let first construct an isomorphism $\theta$ between the boolean complex $\bool{J}$ and the complex $\MM{c,a,b,I,J}_p$. For $p=2$ we can map the generator $e_K$ of $\bool{J}$ associated to a subset $K$ of $J$ to the element $$ \theta(e_k):= \gs (z_c, x_0^a x_1^b x_{i_1}^2 \cdots x_{i_k}^2\epsilon(x_{j_1})\cdots \epsilon(x_{j_h})),$$ where $\epsilon(x_{j_t}) = x_{j_t}$ if $j_t \in K$ and $\epsilon(x_{j_t}) =  x_{j_t-1}^2 $ if $j_t \notin K$. 

For $p>2$ we can map the generator $e_K$ of $\bool{J}$ associated to a subset $K$ of $J$ to the element $$ \theta(e_k):= \gs (z_c, h^a x_1^b y_{i_1} \cdots y_{i_k}\epsilon(x_{j_1})\cdots \epsilon(x_{j_h})),$$ where $\epsilon(x_{j_t}) = x_{j_t}$ if $j_t \in K$ and $\epsilon(x_{j_t}) =  y_{j_t} $ if $j_t \notin K$. 

It is easy to check that $\theta \circ d = \beta_p \circ \theta$. Similarly we can construct an isomorphism $\theta': \bool{J} \to \MMp{c,a,b,I,J}_p$ with
$$
\theta'(e_J):= \gp (z_c, x_0^a x_1^b x_{i_1}^2 \cdots x_{i_k}^2\epsilon(x_{j_1})\cdots \epsilon(x_{j_h}))
$$
for $p=2$ and
$$
\theta'(e_J):= \gp (z_c, h^a x_1^b y_{i_1} \cdots y_{i_k}\epsilon(x_{j_1})\cdots \epsilon(x_{j_h}))
$$
for $p>2$ 
and check that $\theta' \circ d = \beta_2 \circ \theta'$.
\end{proof}

%\todo{continuare anche qui anche con il caso $p$ dispari}
\item 
Given multi-indices $I=(i_1, \ldots, i_k)$ (resp.~$J=(j_1, \cdots, j_l)$)  we write $x_I$ for $x_{i_1} \cdots x_{i_k}$ (resp.~$y_J$ for $y_{j_1} \cdots y_{j_l}$).

We recall from Proposition \ref{prop:generators} that the elements of the form $\gp (z_c, x_0 x_I)$ and $\gs (z_c, x_0 x_I)$ (resp.~$\gp (z_c, h^{r} y_J x_I)$ and $\gs(z_c, h^{r} y_J x_I)$ for $p$ an odd prime) are basis of $H_*(\Gb{n}; \F_2[t]/(1-(-t)^d))$ (resp.~$H_*(\Gb{n}; \F_p[t]/(1-(-t)^d))$ for $p$ an odd prime). We assume this basis fixed.

%Fix a prime $p$.
\item For a given prime $p$ we claim that any two distinct modules $\MM{c,a,b,I,J}_p$ (or $\MMp{c,a,b,I,J}_p$)  
%have disjoint set of generators.
are contained in the span of disjoint subsets of the given basis and in particular they are in direct sum.
%\todo{dimostrare} 
The case of modules of the form $\MM{c,a,b,I,J}_2$ can be proved as follows (the cases of modules of the form $\MMp{c,a,b,I,J}_2$, and of the form  $\MM{c,a,b,I,J}_p$ and  $\MMp{c,a,b,I,J}_p$ when $p$ is odd,
 are analogous).  
Let
$$
\gs_0 = \gs (z_c, x_0^a x_1^b x_{i_1}^2 \cdots x_{i_k}^2\epsilon(x_{j_1})\cdots \epsilon(x_{j_h}))
$$
with $\epsilon(x_{j_t}) = x_{j_t} \mbox{ or }  x_{j_t-1}^2$ be a generator of $\MM{c,a,b,I,J}$. We can choose an element $\gs_1$ of the form 
$$
\gs_1 = \gs (z_c, x_0^a x_1^b x_{i_1}^2 \cdots x_{i_k}^2\epsilon'(x_{j_1})\cdots \epsilon'(x_{j_h}))
$$
such that $\gs_0$ appears as a summand in $\beta_2(\gs_1).$ The constrains on the multi-indexes $I$ and $J$ and formula \eqref{bockstein2} imply that such an element $\gs_1$ exists if and only if for at least one index $l$ we have that $\epsilon(x_{j_l}) = x_{j_l-1}^2$. If such an element $\gs_1$ exists we have that $\gs_1 \in \MM{c,a,b,I,J}$ and we say that $\gs_0$ \emph{lifts} to $\gs_1$. Hence in a finite number of steps we have that $\gs_0$ lifts to $$\gs (z_c, x_0^a x_1^b x_{i_1}^2 \cdots x_{i_k}^2x_{j_1}\cdots x_{j_h}).
$$  This implies that an element $\gs_0$ does not belong at the same time to two different modules
$\MM{c,a,b,I,J}_2$ and $\MM{c',a',b',I',J'}_2.$

Therefore if we fix a prime $p$ the modules $\MM{c,a,b,I,J}_p$ for all admissible $c,a,b,I,J$ are in direct sum and the modules $\MMp{c,a,b,I,J}_p$ for all admissible $c,a,b,I,J$ are in direct sum.
\item Next we claim that every element of the form
$\gp (z_c, x_0 x_{l_1} \cdots x_{l_k})$ or of the form $\gs (z_c, x_0 x_{l_1} \cdots x_{l_k})$ (resp.~$\gp (z_c, h^{r} y_J x_I)$ and $\gs(z_c, h^{r} y_J x_I)$ for $p$ odd) appears in at least one 
complex  $\MM{c,a,b,I,J}_p$ or $\MMp{c,a,b,I,J}_p$. 
%\todo{dimostrare}
Let us prove this when $p=2$, in the case of a generator of the form $\gs (z_c, x_0 x_{l_1} \cdots x_{l_k})$, the other cases being analogous. We can write the monomial $x_0x_{l_1} \cdots x_{l_k}$ as
$$
x_{q_1}^{p_{q_1}} \cdots x_{q_r}^{p_{q_r}} 
$$
with $q_1 < q_2 < \cdots < q_r$. %, where we assume $p_{q_s} > 0$ for all $s$ and we set $p_s = 0 $ if $s \neq q_t$ for all $t$.
Then we define the strictly ordered multi-index $J$ as follows: $j \in J$ if and only if $j>1$ and one of the following conditions is satisfied:
\begin{enumerate}
	\item  $p_j$ is odd;
	\item $p_j=0$ and $p_{j-1}$ is even and non-zero.
\end{enumerate}
Moreover if $p_j$ is odd we set $\epsilon(x_j) = x_j$, otherwise we set $\epsilon(x_j) = x_{j-1}^2$.
Next we define the multi-index $I$ suitably in the unique way such that 
$$
x_{q_1}^{p_{q_1}} \cdots x_{q_r}^{p_{q_r}} = x_0^{p_0} x_1^{p_1} x_{i_1}^2 \cdots x_{i_k}^2\epsilon(x_{j_1})\cdots \epsilon(x_{j_h}).
$$
It is straightforward to check that 
$$
\gs (z_c, x_0^{p_0} x_1^{p_1} x_{i_1}^2 \cdots x_{i_k}^2\epsilon(x_{j_1})\cdots \epsilon(x_{j_h})) = \gs (z_c,  x_{q_1}^{p_{q_1}} \cdots x_{q_r}^{p_{q_r}} ) = 
$$
$$
=\gs (z_c, x_0 x_{l_1} \cdots x_{l_k})
$$
and that the multi-indexes $I$ and $J$ satisfies the condition of Definition \ref{def:modules}.

\item Hence 
%let $\MMM$ be 
$h_{i}(n,p)$ is the direct sum of all admissible modules $\MM{c,a,b,I,J}_p$ and 
%let $\MMMp$ be 
$h_{i}'(n,p)$ is
the direct sum of all admissible modules $\MMp{c,a,b,I,J}_p$ and we recall (equation \eqref{eq:decomposizione}) that
$$
H_*(\Gb{n}; \F_p[t]/(1-(-t)^d)) = h_{i}(n,2) \oplus h_{i}'(n,2).
$$
This direct sum decomposition  implies that the homology $H_{\beta_p}$ of the complex $(H_*(\Gb{n};$ $\F_p[t]/(1-(-t)^d)), \beta_p)$ is given as follows:
$$
H_{\beta_p} = \bigoplus \MM{c,a,b,\emptyset, \emptyset}_p
\oplus
\bigoplus \MMp{c,a,b,\emptyset, \emptyset}_p
$$
for $c$ even, $a \in \N_{>0}$, $b \in \{0,1\}$. In fact if $J = \emptyset$ then also $I = \emptyset$ and for all non-empty $J$ we have from Lemma \ref{lem:boolean} that the complexes $(\MM{c,a,b,I,J}_p,\beta_p)$ and $(\MMp{c,a,b,I,J}_p,\beta_p)$ are acyclic.

\item Using Remark \ref{rem:Jvuoto}  it is easy to check that the complex $H_{\beta_p}$ has Poincar\'e polynomial $(1+t)(1+t+t^2+ \cdots + t^{n-1})$, hence Lemma \ref{lem:no2tor2} follows.
\end{enumerate}
\end{proof}
In \cite{calmar} it was proved that $H_*(\B(2d,d,n); \ring) = H_*(\Gb{n}; \ring[t]/(1-(-t)^d)):$  the homology of the complex braid group $\B(2d,d,n)$ with coefficients in any field was computed, but no computation about the Bockstein homomorphism was given in this case.  
As a consequence of Lemma  \ref{lem:no2tor2} we can add:
\begin{cor}
For $n$ odd the integer homology of $\B(2d,d,n)$ has no $p^2$-torsion.
\end{cor}

\begin{lem} \label{lem:tau_torsion} %\todo{adattare a $\confmd{n}^d$}
Let $\ring = \Z$. Let $p$ be a prime and assume that $p^k \mid d$, but $p^{k+1} \nmid d$. For $n$ odd the cokernel of the homomorphism
$$\tau:H_i(\confm{n}) \to H_i(\confmd{n}^d)
$$
has no $p^{k+1}$-torsion.
%	\todo{per $d$ dispari questo vale anche per $n$ pari?}
\end{lem}
\begin{proof} The proof is practically identical to \cite[Lem.~5.4]{cal_sal_superelliptic}: we use here the sets $\cB' \subset H_i(\confm{n};\Z)$  and $\cB'' \subset H_i(\confmd{n}^d;\Z)$ introduced in Definition \ref{def:BB}.  These are free sets of generators of a maximal free $\Z$-submodule of $H_i(\confm{n};\Z)$  and $H_i(\confmd{n}^d;\Z)$ respectively (Proposition \ref{prop:BB}).

The map $\tau$ 
acts on the elements of $\cB'$ by:
\begin{eqnarray}
\tau: \omega_{2i,j,0}^{(1)} & \mapsto  \frac{1-(-t)^d}{1+t} &\omega_{2i,j,0}^{(d)} = d \omega_{2i,j,0}^{(d)}, \label{rat_gen1}\\
\tau: \widetilde{\omega}_{2i,j,0}^{(1)} & \mapsto \phantom{\frac{1-(-t)^d}{1+t}}& \widetilde{\omega}_{2i,j,0}^{(d)},\label{rat_gen2}\\
\tau: \omega_{2i,j,1}^{(1)}  & \mapsto \frac{1-(-t)^d}{1+t}& \omega_{2i,j,1}^{(d)} = d \omega_{2i,j,1}^{(d)},\label{rat_gen3}\\
\tau: \widetilde{\omega}_{2i,j,1}^{(1)} & \mapsto \phantom{\frac{1-(-t)^d}{1+t}}& \widetilde{\omega}_{2i,j,1}^{(d)}.\label{rat_gen4}
\end{eqnarray}
Therefore $\tau$ diagonally maps each element of $\cB'$ to $1$ or $d$ times the corresponding element of $\cB''$.

Then the result follows since the
 $\Z$-modules $H_i(\confm{n};\Z)$ and $H_i(\confmd{n}^d;\Z)$ have no $p^2$-torsion and  $\tau$ is an isomorphism mod $p$ if $p \nmid d$ (see Lemma \ref{lem:tau_invertible}).
\end{proof}

Now we consider homology with coefficient in $\ring = \Z$. As stated in 
%Corollary \ref{cor:image_mu}
\cite[Cor.~4.10]{cal_sal_superelliptic}, the image of  $\mu_*$ is the submodule $H_i(\confm{n}, \conf_n)$ in $H_i(\confm{n})$ and we consider the composition $\iota = J \circ \tau \circ \mu_*: H_{i-1}(\confm{n-1}) \otimes H_1(S^1) \to H_i(\confmd{n}^d, \conf_n)$.

\begin{lem} \label{lem:comp_no_4}
% aggiunta nel caso generale
Let $p$ be a prime and assume that $p^k \mid d$, but $p^{k+1} \nmid d$. 
For $n$ odd 
%and $(n,d)=1$ 
the cokernel of the composition $J \circ \tau \circ \mu_*$ has no $p^{k+1}$-torsion.
% \todo{per $d$ dispari questo vale anche per $n$ pari?}
\end{lem}
The proof of the lemma above is analogous to the proof of \cite[Lem.~5.5]{cal_sal_superelliptic}, with suitable modification. For the reader's convenience we give here the adapted proof.
\begin{proof}
%\todo{adatta come sopra}
First we can consider the homomorphism $\overline{s}_*:H_*(\conf_n) \to H_*(\confm{n})$ induced by the inclusion $\overline{s}:\conf_n \into \confm{n}$; given the decomposition $H_*(\confmd{n}^d) = H_*(\conf_n) \oplus H_*(\confmd{n}^d, \conf_n)$, we call $\pi_1$ and $\pi_2= J$ respectively the projections of $H_*(\confmd{n}^d)$ onto the first and the second summand,  and hence the map  $\tau_{11}:H_*(\conf_n) \to H_*(\conf_n)$ defined by the composition
$$
H_*(\conf_n) \stackrel{\overline{s}_*}{\longrightarrow} H_*(\confm{n}) \stackrel{\tau}{\longrightarrow} H_*(\confmd{n}^d) \stackrel{\pi_1}{\longrightarrow} H_*(\conf_n).
$$
We can consider the following diagram:
$$
\xymatrix{ H_*(\conf_n) \ar[d]^{\overline{s}_*} \ar[drr]^(0.65){
%s_*+s'_*
\sum_{j=0}^d s^{(j)}_*
} \ar[rr]_(0.44){\tau_{|\overline{s}_*H_*(\conf_n)}} \ar@/^1pc/[rrr]^{\tau_{11}}
	&& H_*(
\bigsqcup_{j=0}^{d-1}	s^{(j)}(\conf_{n})  
%\sqcup s'(\conf_{n})
	)\ar[d]^{i_*} \ar[r] & H_*(\conf_n) \\ H_*(\confm{n}) \ar[rr]^\tau && H_*(\confmd{n}^d) \ar[ur]^{\pi_1}  }
$$
From Theorem \ref{teo:tau_restricted} we have that that $\tau_{11} = d \Id_{H_*(\conf_n)}$ and
%Moreover we claim that 
$\pi_2 \tau\overline{s}_*(H_*(\conf_n))=0$.
%This follows since $s_*$ is the left inverse of $\pi_1$ and hence for $x \in H_*(\conf_n)$ we have that $\pi_2(\tau(\overline{s}_*(x))) = \pi_2(s_*(x) + s'_*(x) ) = 2 \pi_2(s_*(x))=0$.

Now, let $x_2 \in H_*(\confmd{n}^d, \conf_n)$ 
%be an element such that $px_2 =0$ 
and let $\tau_{22}: H_*(\confm{n}, \conf_n) \to H_*(\confmd{n}^d, \conf_n)$ be the map induced by $\tau$ by restricting to $H_*(\confm{n}, \conf_n)$ and projecting to $H_*(\confmd{n}^d, \conf_n)$. Recall from 
%Proposition \ref{prop:image_mu} 
\cite[Prop.~4.2]{cal_sal_superelliptic}
that $\mu_*$ is injective. If there exists $y \in H_*(\confm{n})$ such that $\pi_2 \tau (y) = p^{k+1} x_2$, then let $x_1:= \pi_1(\tau(y))$. We can consider $-x_1 + d y \in H_*(\confm{n})$ and we have that $\tau(-x_1 + d y) = - d x_1 + d (x_1 +p^{k+1}x_2) = dp^{k+1} x_2$. Since the cokernel of $\tau$ has at most $p^k$-torsion (see Lemma \ref{lem:tau_torsion}) it follows that $dx_2 = 0$ in $\coker \tau$ and finally, since $\pi_2 \tau \overline{s}_*(H_*(\conf_n))= 0$, $p^k x_2 = 0$ in $\coker \tau_{22}$.
\end{proof}

%hence the image is included in the kernel of the projection to $H_i(\conf_n)$. The homomorphism $\tau$ maps  $H_i(\confm{n}, \conf_n)$ to the submodule $H_i(\confmd{n}, \conf_n)$ of  $H_i(\confmd{n})$, because this is the kernel of the projection to $H_i(\conf_n)$.
%Hence, since from Lemma \ref{lem:tau_torsion}  the cokernel of the homomorphism $\tau$ has no $4$-torsion, the same holds for the cokernel of the composition $J \circ \tau \circ \mu_*: H_{i-1}(\confm{n-1}) \otimes H_1(S^1) \to H_i(\confmd{n}, \conf_n)$.

From 
%Lemma \ref{lem:no2tor1} 
\cite[Lem.~5.1]{cal_sal_superelliptic}
and \ref{lem:comp_no_4} we have that, with integer coefficients, the kernel  of the map $$\iota:H_{i-1}(\Br_n; H_1(S^1 \times P)) \to H_{i-1}(\Br_n; H_1(\ddiskP^d))$$ in diagram \eqref{eq:bia2} has no $p^2$-torsion and the cokernel of $\iota$ has no $p^{k+1}$-torsion. 
Hence we have
% (for $d=2$ see also \cite{bianchi}):
:
\begin{thm}\label{thm:4tor}
Let $p$ be a prime and assume that $p^k \mid d$, but $p^{k+1} \nmid d$. 
For $n$ odd 
%and $(n,d)=1$  
the homology $H_{i}(\Br_n; H_1(\surf_n^d))$ computed with coefficients in the ring $\ring = \Z$ has $p$-torsion of order at most $p^{k+1}$. \qed
\end{thm}
%The map $\tau$, restricted to the image of $\mu_*$, maps suitable generators of the $\Z$-module $H_i(\confm{n})$ to $1$ or $2$ times suitable generators of the $\Z$-module $H_i(\confmd{n})$.

\section{No $p^{k+1}$-torsion}\label{sec:no4tor}
In this section we will show that 
%\todo{serve anche qui $(n,d)=1$?}
if $p$ is a prime such that $p^k \mid d$ ($k>0$), but $p^{k+1} \nmid d$
and $n$ is odd then $H_{i}(\Br_n; H_1(\surf_n^d))$ has $p$-torsion of order at most $p^k$. In particular for all this section we will assume that $p$ is a prime that divides $d$.

%In order to prove this 
We will consider 
for $n$ odd the following short exact sequence associated to \eqref{eq:bia2}, with coefficients in $\Z_p$ and in $\Z$
\begin{equation}\label{eq:short_split}
0 \to \coker \iota \to H_{i}(\Br_n; H_1(\surf_n^d)) \to \ker \iota \to 0.
\end{equation}
\begin{prop}\label{prop:split}
The %$p$-torsion part of the 
short exact sequence \eqref{eq:short_split} with $\Z$ coefficients splits.
\end{prop}
\begin{proof}
Let us fix an odd integer $n$. We write $$H_{i}(\Br_n; H_1(\surf_n^d; \Z)) \otimes \Z_p = \Z_p^{a_i},$$
$$\coker (\iota_i :H_{i}(\Br_n; H_1(S^1 \times P;\Z)) \to H_{i}(\Br_n; H_1(\ddiskP^d;\Z) )\otimes \Z_p = \Z_p^{u_i}$$ and $$\ker (\iota_i :H_{i}(\Br_n; H_1(S^1 \times P;\Z)) \to H_{i}(\Br_n; H_1(\ddiskP^d;\Z) )\otimes \Z_p = \Z_p^{v_i}$$

Since for $n$ odd we have seen that $H_{i}(\Br_n; H_1(\surf_n^d; \Z))$ is all torsion (see Theorem \ref{th:only_d_torsion}), we have that the $p$-torsion part of \eqref{eq:short_split} splits if and only if 
$$
u_i + v_{i-1} = a_i.
$$

%Since the modules $\ker \iota$ and $\coker \iota$ have no $4$-torsion, we can assume $$\coker (\iota_i :H_{i}(\Br_n; H_1(S^1 \times P;\Z)) \to H_{i}(\Br_n; H_1(\ddiskP;\Z) ) = \Z_2^{u_i}$$ and $$\ker (\iota_i :H_{i}(\Br_n; H_1(S^1 \times P;\Z)) \to H_{i}(\Br_n; H_1(\ddiskP;\Z) ) = \Z_2^{v_i}$$
%and clearly we have 
%$$
%u_i + v_{i-1} = a_i + 2 b_i.
%$$

Moreover, with coefficients in $\Z_p$, we have
$$
H_{i}(\Br_n; H_1(\surf^d_n; \Z_p)) = \Z_p^{a_i+a_{i-1}}.
$$
Let $$\overline{u}_i := \rk \coker (\iota_i :H_{i}(\Br_n; H_1(S^1 \times P ;\Z_p)) \to H_{i}(\Br_n; H_1(\ddiskP^d ;\Z_p) ) $$
and 
$$\overline{v}_i := \rk \ker (\iota_i :H_{i}(\Br_n; H_1(S^1 \times P ;\Z_p)) \to H_{i}(\Br_n; H_1(\ddiskP^d;\Z_p) ). $$
It follows that
the $p$-torsion part of \eqref{eq:short_split} splits if and only if 
$$2 \sum_i (u_i + v_i) = \sum_i (\overline{u}_i  + \overline{v}_i ). $$

Hence we can compute the rank of the modules above.
%\todo{adattare i generatori per $p$ qualsiasi che divide $d$}

A basis of the homology $H_{i}(\Br_n; H_1(S^1 \times P;\Z_p))$ is given as follows. 
Following \cite{calmar}, the homology $H_*(\Gb{n}; \F_p[t]/(1+t))$ for $n$ odd is generated, as an $\F_p[t]$-module, by the classes of the following form (see Proposition \ref{prop:generators}):
%\todo{adattare}
\begin{enumerate}
	\item[a)] for $p=2$,
\begin{equation}\label{generators1Z2}
\gpu (z_c, x_0 x_{i_1} \cdots x_{i_k}) :=z_{c+1} x_{i_1} \cdots x_{i_k}
\end{equation}
 and
\begin{equation}\label{generators2Z2}
\gsu(z_c, x_0 x_{i_1} \cdots x_{i_k}):=\frac{\DB(z_{c+1}x_{i_1} \cdots x_{i_k})}{(1+t)}
\end{equation}
where we assume $0 \leq i_1 \leq \cdots i_k$ and $c$ even;
\item[b)] for $p$ an odd prime,
\begin{equation}
\gpu (z_c, h^{r} y_{j_1} \cdots y_{j_l } x_{i_1} \cdots x_{i_k}) :=  z_{c+1} h^{r-1} y_{j_1} \cdots y_{j_l }x_{i_1} \cdots x_{i_k}.
\end{equation}
and
\begin{equation}
\gsu(z_c, h^{r} y_{j_1} \cdots y_{j_l } x_{i_1} \cdots x_{i_k}):=\frac{\DB(z_{c+1}h^{r-1} y_{j_1} \cdots y_{j_l }x_{i_1} \cdots x_{i_k})}{(1+t)}.
\end{equation}

\end{enumerate}

In particular for $p=2$ the image of $\mu_*$ if generated by all elements of the form $\gpu (z_c, x_0 x_{i_1} \cdots x_{i_k})$ and all elements of the form  $\gsu(z_c, x_0 x_{i_1} \cdots x_{i_k})$ with $c>0$, while for $p$ odd the image of $\mu_*$ is  generated by all elements of the form  $\gp_1 (z_c, h^{r} y_{j_1} \cdots y_{j_l } x_{i_1} \cdots x_{i_k})$ and all  elements of the form  $\gs_1(z_c, h^{r} y_{j_1} \cdots y_{j_l } x_{i_1} \cdots x_{i_k})$ with $c>0$.

As seen in Section 
%\ref{s:4tor} 
\ref{sec:homol_artin}
the homology $H_*(\Gb{n}; \F_p[t]/(1-(-t)^d))$ for $n$ odd is generated, as an $\F_p[t]$-module, by the following classes
\begin{enumerate}
	\item[a)]
for $p=2$, by the classes (already introduced in  \eqref{generators1} and \eqref{generators2}, Proposition \ref{prop:generators}):
\begin{equation*}
\gp (z_c, x_0 x_{i_1} \cdots x_{i_k}) :=\frac{1-(-t)^d}{1+t}z_{c+1} x_{i_1} \cdots x_{i_k}
\end{equation*}
and
\begin{equation*}
\gs(z_c, x_0 x_{i_1} \cdots x_{i_k}):=\frac{\DB(z_{c+1}x_{i_1} \cdots x_{i_k})}{(1+t)}
\end{equation*}
where we assume $0 \leq i_1 \leq \cdots i_k$ and $c$ even;
\item[b)] for $p$ odd, by the classes (already introduced in
\eqref{generators1p}, \eqref{generators2p}, Proposition \ref{prop:generators}
\begin{equation}\label{generators1p_second}
\gp (z_c, h^{r} y_{j_1} \cdots y_{j_l } x_{i_1} \cdots x_{i_k}) := \frac{1-(-t)^d}{1+t}  z_{c+1} h^{r-1} y_{j_1} \cdots y_{j_l }x_{i_1} \cdots x_{i_k}.
\end{equation}
and
\begin{equation}\label{generators2p_second}
\gs(z_c, h^{r} y_{j_1} \cdots y_{j_l } x_{i_1} \cdots x_{i_k}):=\frac{\DB(z_{c+1}h^{r-1} y_{j_1} \cdots y_{j_l }x_{i_1} \cdots x_{i_k})}{(1+t)}.
\end{equation}
\end{enumerate}

%The kernel of $J$ is generated \todo{qui cosa si sta usando?}:
The following classes generate the image of $s_*$ and hence these are the generators of the kernel of $J$:
\begin{enumerate}
	\item[a)] for $p=2$, 
by the classes 
$\gs(z_c, x_0 x_{i_1} \cdots x_{i_k})$ for $c=0$;
\item[b)] for $p$ odd, by the classes $\gs(z_c, h^{r} y_{j_1} \cdots y_{j_l } x_{i_1} \cdots x_{i_k})$ with $c=0$
\end{enumerate}

Hence for any $p$ the map $\tau$ acts as follows:
\begin{enumerate}
	\item[a)] for $p=2$:
\begin{eqnarray*}
\tau: \gpu (z_c, x_0 x_{i_1} \cdots x_{i_k}) & \mapsto & \phantom{d} \gp (z_c, x_0 x_{i_1} \cdots x_{i_k});\\
\tau: \gsu (z_c, x_0 x_{i_1} \cdots x_{i_k}) &\mapsto  &   d \gs (z_c, x_0 x_{i_1} \cdots x_{i_k});
\end{eqnarray*}
\item[b)] for $p$ odd:
\begin{eqnarray*}
	\tau: \gpu (z_c, h^{r} y_{j_1} \cdots y_{j_l } x_{i_1} \cdots x_{i_k}) & \mapsto & \phantom{d} \gp (z_c, h^{r} y_{j_1} \cdots y_{j_l } x_{i_1} \cdots x_{i_k});\\
	\tau: \gsu (z_c, h^{r} y_{j_1} \cdots y_{j_l } x_{i_1} \cdots x_{i_k}) &\mapsto  & d \gs (z_c, h^{r} y_{j_1} \cdots y_{j_l } x_{i_1} \cdots x_{i_k}).
\end{eqnarray*}
\end{enumerate}
Hence if we assume that $p \nmid d$ we have that $\iota$ is an isomorphism with coefficients in $\Z_p$.
As a consequence using the Universal Coefficients Theorem we can state a stronger version of Theorem \ref{th:only_d_torsion}.

Let now focus on a prime $p$ that divides $d$. 
\begin{rem}\label{rem:basi_mod_2}
%	\todo{per $d$ dispari questo vale anche per $n$ pari?}
Let $p$ a prime such that $p\mid d$ and let $n$ be an odd integer. A basis of $$\coker (\iota_i :H_{i}(\Br_n; H_1(S^1 \times P;\Z_p)) \to H_{i}(\Br_n; H_1(\ddiskP^d;\Z_p) )$$ is
given:
%\todo{da specificare $n$ dispari?}
\begin{enumerate}[wide=0pt]
	\item[a)] for $p=2$,
 by the elements $\gs (z_c, x_0 x_{i_1} \cdots x_{i_k})$ with $c$ even, $c>0$, of degree $n$ and homological dimension $i+1$; 
 	\item[b)] for $p$ odd,
 by the elements $\gs(z_c, h^{r} y_{j_1} \cdots y_{j_l } x_{i_1} \cdots x_{i_k})$ with $c$ even, $c>0$, $r>0$, of degree $n$ and homological dimension $i+1$; 
\end{enumerate}
A basis of $$ \ker (\iota_i :H_{i}(\Br_n; H_1(S^1 \times P;\Z_p)) \to H_{i}(\Br_n; H_1(\ddiskP^d;\Z_p) )$$ is given:
\begin{enumerate}[wide=0pt]
	\item[a)] for $p=2$, 
 by the elements $\gsu (z_c, x_0 x_{i_1} \cdots x_{i_k})$ with $c$ even, $c>0$, of degree $n$ and homological dimension $i+1$;
\item[b)] for $p$ odd
 by the elements $\gsu (z_c, h^{r} y_{j_1} \cdots y_{j_l } x_{i_1} \cdots x_{i_k})$ with $c$ even, $c>0$, $r>0$, of degree $n$ and homological dimension $i+1$.
\end{enumerate}
Clearly in both cases a) and b) the base of the kernel is in bijection with the corresponding base of the cokernel  and we have $\overline{u}_i = \overline{v}_i$.

\end{rem}
 
%$\rk \coker (\iota_i :H_{i}(\Br_n; H_1(S^1 \times P);\Z_2) \to H_{i}(\Br_n; H_1(\ddiskP);\Z_2) ) $
%$\rk \ker (\iota_i :H_{i}(\Br_n; H_1(S^1 \times P);\Z_2) \to H_{i}(\Br_n; H_1(\ddiskP);\Z_2) ). $

Now recall (see Definition \ref{def:BB} and Proposition \ref{prop:BB}) the bases
 $\cB'$, $\cB''$ generating the homology of $H_i(\confm{n};\Q) $ and $H_i(\confmd{n}^d;\Q)$ and spanning a maximal free $\Z$-submodule of $H_i(\confm{n};\Z) $ and $H_i(\confmd{n}^d;\Z)$. Recall also the action of $\tau$  over these bases given in equations (\ref{rat_gen1}--\ref{rat_gen4}). The elements of $\cB'$ (resp.~$\cB''$) of the form $\omega_{2i,j,\epsilon}^{(1)}$ (resp.~$\omega_{2i,j,\epsilon}^{(d)}$) map, modulo $p$, to elements of the form $\gsu$ (resp.~$\gs$) and in particular the elements the form $\omega_{2i,j,\epsilon}^{(1)}$ (resp.~$\omega_{2i,j,\epsilon}^{(d)}$) with $i=0$ map to elements of the form $\gsu(z_c, \ldots)$ (resp.~$\gs(z_c, \ldots)$) with $c=0$.
The elements of $\cB'$ (resp.~$\cB''$) of the form $\widetilde{\omega}_{2i,j,\epsilon}^{(1)}$ (resp.~$\widetilde{\omega}_{2i,j,\epsilon}^{(d)}$) map, modulo $p$, to elements of the form $\gpu$ (resp.~$\gp$).
Let $$w_i = | \{ \omega_{0,j,\epsilon}^{(1)} \in H_i(\confm{n};\Q)\} | = 
| \{ \omega_{0,j,\epsilon}^{(p)} \in H_i(\confmd{n}^d;\Q)\} |.
$$
From the Universal Coefficients Theorem and from the description of $\tau$ given in equations (\ref{rat_gen1}--\ref{rat_gen4}) we have that
$$
\sum_i u_i = \sum_i  \frac{ \overline{u}_i - w_i}{2} +\sum_i  w_i
$$
and
$$
\sum_i v_i =  \sum_i  \frac{\overline{v}_i - w_i}{2}.
$$
Then it is straightforward to see that 
$$2 \sum_i (u_i + v_i) = \sum_i (\overline{u}_i  + \overline{v}_i ) $$
and this imply that the exact sequence \eqref{eq:short_split} splits and Proposition \ref{prop:split} follows. 
\end{proof}
As a corollary of Proposition \ref{prop:split},  under the hypothesis of Theorem \ref{thm:4tor}  we have the following result.
\begin{thm} \label{th:no4tor}
%	\todo{forse si pu\`o migliorare eliminando l'ipotesi $(n,d)=1$?}
For odd $n$ %and $(n,d)=1$ 
% \todo{forse vale anche per $n$ even e $d$ odd?}
the homology $H_{i}(\Br_n; H_1(\surf^d_n;\Z))$ has no torsion of order $p^k$ if $p^k \nmid d$. \qed
\end{thm}
\begin{rem}\label{rem:dn}
For $d=2$ and $n$ odd Theorem \ref{th:no4tor} allows to compute the homology 
$$
H_i(\totsp_n^2; \Z) = H_{i}(\Br(n); \Z) \oplus H_{i-1}(\Br_n; H_1(\surf_n^2;\Z))
$$
that corresponds (see Theorem \ref{thm:Bddn}) to the homology of the Artin group of type $D_n$ (see also \cite{gorj}).
\end{rem}

\section{Natural maps between families}
Let $d, d'$ two positive integers such that $d \mid d'$ and $d' = m d$.  There is a natural ramified covering
$$
\surf_n^{d'} \to \surf_n^d
$$
given by mapping $(z,y) \mapsto (z,y^{m})$ and hence a corresponding map between the spaces
$$
\rho = \rho_n^{(d',d)} : \totsp_n^{d'} \to \totsp_n^d
$$
given by $(P,z,y) \mapsto (P,z,y^{m})$.

%$(\rho_n^{(d',d)})_* : H_*(\totsp_n^{d'}) \to H_*(\totsp_n^d)$. 
The map $\rho = \rho_n^{(d',d)}$ commutes with the projections $\pi: \totsp_n^{d'} \to \conf_n$ and $\pi: \totsp_n^{d'} \to \conf_n$ and with the  sections $s:\conf_n \to \totsp_n^{d'}$ and $s:\conf_n \to \totsp_n^{d}$. 

In this section we investigate the induced homology map 
$$
\rho_* = (\rho_n^{(d',d)})_* : H_*(\totsp_n^{d'}, \conf_n) \to H_*(\totsp_n^d, \conf_n).
$$
\begin{rem}
The map $\rho = \rho_n^{(d',d)}$ commutes with the ramified covering $\totsp^{d'}_n \stackrel{p}{\longrightarrow} \conf_n \times \disk 
$ and $\totsp^d_n \stackrel{p}{\longrightarrow} \conf_n \times \disk $ preserving their ramification loci. Hence $\rho$ induces a morphism between the corresponding long exact sequences \eqref{eq:bia2} for $d'$ and $d$, giving a commuting diagram
\begin{equation}\label{eq:bia3}
\xymatrix  @R=1.5pc @C=1.5pc {
\cdots
\ar[r] &  H_{i-1}(\confm{n-1}) \otimes H_1(S^1) \ar[d]^m
\ar[r]^(0.6)\iota &  H_{i}(\confmd{n}^{d'}, \conf_n) \ar[r] \ar[d]^{\overline{\rho}_*} &   H_{i-1}(\Br_n; H_1(\surf_n^{d'})) \ar[r] \ar[d]^{\rho_*} &  
\cdots \\
\cdots \ar[r] &  H_{i-1}(\confm{n-1}) \otimes H_1(S^1)
\ar[r]^(0.6)\iota &  H_{i}(\confmd{n}^d, \conf_n) \ar[r] &   H_{i-1}(\Br_n; H_1(\surf_n^d)) \ar[r] &  \cdots
}
\end{equation}
where the map $\overline{\rho}_*$ is induced by the restriction of $\rho$ to the space $\confmd{n}^{d'}$: $$\overline{\rho}:\confmd{n}^{d'} \to \confmd{n}^{d}.$$
Notice that the vertical arrow on the left is induced by the restriction of $\rho$ to a tubular neighborhood of the ramification locus. In particular the restriction induces an $m$-fold covering $S^1 \to S^1$ and hence the map induced in homology is the multiplication by $m = \frac{d'}{d}$.
\end{rem}

\begin{lem}\label{lem:coverings}
Consider homology with coefficients in a field $\F$ of characteristic $p$. Assume that $2 \mid d$ 
%and $2 \nmid m$ 
or that $2 \nmid d'$. Moreover assume that $p \nmid m$. Then the map $\overline{\rho}$ induces an isomorphism
$$
\overline{\rho}_*: H_*(\confmd{n}^{d'}, \conf_n; \F) \to H_*(\confmd{n}^{d}, \conf_n; \F).
$$
\end{lem}
\begin{proof}
We recall that we can consider the ring $\F[t^\pmu]$ as a module over the group $\Gb{n}$ as with the usual action (see \cite[\S 3]{cal_sal_superelliptic}).
Since we have the relation $$1-(-t)^{d'} = (1-(-t)^d)[m]_{(-t)^d}$$ we can consider the following commuting diagram of $\Gb{n}$-modules:
$$
\xymatrix{
0 \ar[r] & \F[t^{\pmu}] \ar[d]^{[m]_{(-t)^d}}
\ar[r]^{1-(-t)^{d'}} & \F[t^{\pmu}] \ar[d]^{\Id} \ar[r] & \F[t]/({1-(-t)^{d'}}) \ar[r]
\ar[d]^\eta
& 0 \\
0 \ar[r] & \F[t^{\pmu}] \ar[r]^{1-(-t)^{d}} & \F[t^{\pmu}] \ar[r] & \F[t]/({1-(-t)^{d}}) \ar[r] & 0
}
$$
where $\eta$ is the quotient map. The induced diagram of homology groups is
$$
\xymatrix{
	\ldots \ar[r] & H_i(\confm{n}; \F[t^{\pmu}]) \ar[d]^{[m]_{(-t)^d}}
	\ar[r]^{1-(-t)^{d'}} & H_i(\confm{n};\F[t^{\pmu}]) \ar[d]^{\Id} \ar[r] & H_i(\confmd{n}^{d'};\F) \ar[r]
	\ar[d]^{\overline{\rho}_*}
	& \ldots \\
	\ldots \ar[r] & H_i(\confm{n};\F[t^{\pmu}]) \ar[r]^{1-(-t)^{d}} & H_i(\confm{n};\F[t^{\pmu}]) \ar[r] & H_i(\confmd{n}^{d};\F) \ar[r] & \ldots
}
$$
We recall from \cite[Thm.~4.5, 4.12]{calmar} that the group $H_i(\confm{n};\F[t^{\pmu}])$ is a direct sum of $\F[t^{\pmu}]$-modules of the form $\F[t]/(1+t)$ and $\F[t]/(1-t^2)^{(p-1)p^a}$.

Assume that $2 \mid d$. Then if $I = (1+t)$ or $I=(1-t^2)^k$ for some $k>0$ then $[m]_{(-t)^d} \equiv m \mod I$ and hence, since $p \nmid m$, $[m]_{(-t)^d}$ is invertible in $\F[t]/I$. From the five Lemma it follows that the map $\overline{\rho}_*$ is an isomorphism.

Now let us assume that $2 \nmid d'$, and hence $2 \nmid m$. If $p=2$ then $H_i(\confm{n};\F[t^{\pmu}])$ is a direct sum of $\F[t^{\pmu}]$-modules of the form $\F[t]/(1+t)^k$. As seen before, $[m]_{(-t)^d}$ is invertible in $\F[t]/(1+t)^k$ and the same argument gives that $\overline{\rho}_*$ is an isomorphims.

If $p \neq 2$ we can split the quotients $\F[t]/(1-t^2)^{(p-1)p^a}$ as a direct sum
$$
\F[t]/(1-t^2)^{(p-1)p^a} \simeq \F[t]/(1-t)^{(p-1)p^a} \oplus \F[t]/(1+t)^{(p-1)p^a}.
$$
Since we are assuming that $d$ and $p$ are odd the polynomials $1-(-t)^d$ and $1-(-t)^{d'}$ are invertible in $\F[t]/(1-t)^{(p-1)p^a}$. Hence if we restrict to the modules of the form $\F[t]/(1-t)^{(p-1)p^a}$  the kernel and the cokernel of the horizontal maps corresponding to multiplications by $1-(-t)^d$ and $1-(-t)^{d'}$ are trivial. On the other hand if we restrict to the modules of the form $\F[t]/(1+t)^{(p-1)p^a}$ we have that $[m]_{(-t)^d}$ is invertible. Again we obtain   that $\overline{\rho}_*$ is an isomorphism.

In all cases since $\overline{\rho}_*$ restricts to the identity on $H_*(\conf_n, \F)$, it is straightforward to see that the  restriction of $\overline{\rho}_*$ to $H_*(\confmd{n}^{d}, \conf_n; \F)$ is an isomorphism as well.
\end{proof}

\begin{thm}\label{thm:coverings}
Consider homology with coefficients in a field $\F$ of characteristic $p$. Let $d' = dm$. Assume that $2 \mid d$ 
%and $2 \nmid m$ 
or that $2 \nmid d'$. Moreover assume that $p \nmid m$. Then the map $$\rho_*: H_*(\totsp_n^{d'}, \conf_n) \to H_*(\totsp_n^{d}, \conf_n) $$
is an isomorphism.
\end{thm}
\begin{proof}
The theorem follows from Lemma \ref{lem:coverings} and from the five lemma applied to the diagram of long exact sequences \eqref{eq:bia3}.
\end{proof}
\begin{cor}\label{cor:coverings}
Let $d$ be a positive integer and $p$ a prime. 
\begin{enumerate}[label=\alph*), wide=0pt]
\item If $p=2$ and $d=2^am$, with $m$ odd, then
$$H_*(\totsp_n^d, \conf_n; \F_2)=H_*(\totsp_n^{2^a}, \conf_n; \F_2);$$ in particular if $d$ is odd then $H_*(\totsp_n^d, \conf_n; \F_2)=H_*(\totsp_n^{1}, \conf_n; \F_2)=0;$
\item If $p>2$ and $d=p^a m$ with $p \nmid m$ and $m$ odd then
$$H_*(\totsp_n^d, \conf_n; \F_p)=H_*(\totsp_n^{p^a}, \conf_n; \F_p);$$ in particular if $p\nmid d$ and $d$ is odd then 
$H_*(\totsp_n^d, \conf_n; \F_p)=H_*(\totsp_n^{1}, \conf_n; \F_p)=0;$
\item If $p>2$ and $d=p^a m$ with $p \nmid m$ and $m$ even then
$$H_*(\totsp_n^d, \conf_n; \F_p)=H_*(\totsp_n^{2p^a}, \conf_n; \F_p).$$
\end{enumerate}
%Let $d = 2^a m$, with $m$ odd. Let $p$ be a prime such that $p \nmid m$. Then if $a = 0$ (that  is $d$ is odd) $$H_*(\totsp_n^d, \conf_n; \F_p)=H_*(\totsp_n^1, \conf_n; \F_p)=0$$ while if $a>0$ (that is if $d$ is even) and $$H_*(\totsp_n^d, \conf_n; \F_p)= H_*(\totsp_n^{2^a}, \conf_n; \F_p).$$
\end{cor}
\begin{proof}
The statement follows from Theorem \ref{thm:coverings} applied to $\totsp_n^d$ and $\totsp_n^{p^a}$ or $\totsp_n^{2p^a}$. For the case when $d$ is odd we also note that the fibration $\totsp_n^{1} \to \conf_n$ has fiber $\surf_n^1 = \disk$, which is contractible. Hence the spaces $\totsp_n^{1}$ and $\conf_n$ are homotopy equivalent and $H_*(\totsp_n^1, \conf_n; \F_p)=0$ for any prime $p$.
\end{proof}
\begin{rem}
Corollary \ref{cor:coverings} and Theorem \ref{th:no4tor} imply that when $d$ is odd or $n$ is odd the group
$H_*(\totsp_n^d, \conf_n; \Z)$ is finite and has $p$-torsion only if $p \mid d$.

Notice that when $n,d$ are both even the group $H_*(\totsp_n^d, \conf_n; \Z)$ can have $p$-torsion for $p \nmid d$ (see Tables \ref{tab:conti}, \ref{tab:conti3}, \ref{tab:conti5}).
\end{rem}

\begin{rem}
Notice that the isomorphisms given in Theorem \ref{thm:coverings} and Corollary \ref{cor:coverings} are induced by the map $\rho$, hence they are natural and commute with homology operations. It follows from the Bockstein spectral sequence (see \cite[Thm.~10.3, 10.4]{mccleary}) that the isomorphisms above induce isomorphisms for homology with integer coefficients localized to the prime $p$.
\end{rem}

\section{Stabilization and computations}\label{sec:stab}

\subsection{Stabilization results}
Applying the results of Wahl and Randal-Willians (\cite{W-RW}) for the stability of family of groups with twisted coefficients  it is possible to prove  that the groups $H_i(\Br_n; H_1(\surf_n^d))$ stabilize for all $i$. 
In particular the map $$H_i(\Br_n; H_1(\surf_n^d)) \to H_i(\Br_{n+1}; H_1(\surf_{n+1}^d))$$ is an epimorphism for $i \leq \frac{n}{2}-1$ and an isomorphism for $i \leq \frac{n}{2}-2$.
In this section  we will prove a slightly sharper result using the explicit description of the homology.

We recall from \cite[Def.~7]{cal_sal_superelliptic} that the stabilization map $\stab: \confm{n} \to \confm{n+1}$ is defined by 
	$$
	(p_1, \{p_2, \ldots, p_{n+1}\}) \mapsto (p_1, \{p_2, \ldots, p_{n+1}, \frac{1+\max_{1 \leq i \leq n+1} (|p_i|)}{2} \}). 
	$$

%In the notation of \cite{calmar} we have $H_i(\Gb{n}; \F[t]/(1+t)))$ = $H_i(\mathrm{B}(2,1,n); \F))$ and $H_i(\Gb{n}; \F[t]/(1-t^2))) = H_i(\mathrm{B}(4,2,n); \F))$.
We recall that $\confm{n}$ is a classifying space for $\Gb{n}$, that is the Artin group of type $\mathrm{B}$. Moreover from Shapiro Lemma (see \cite{brown}) we have the isomorphism
$$
H_*(\B(2d,d,n);R) = H_*(\Gbn; R[t]/(1-(-t)^d)
$$
(see \cite{calmar} for more details).
Finally we recall from \cite[Cor.~4.17, 4.18,4.19]{calmar} the following result: %(notations $H_i(\mathrm{B}(2,1,n); \F))$ and $H_i(\mathrm{B}(4,2,n); \F))$ were there used respectively for $H_i(\Gb{n}; \F[t]/(1+t))$ and $H_i(\Gb{n}; \F[t]/(1-t^2))$):
\begin{prop}\label{prop:stab_rank}
Let $p$ be a prime or $0$. Let $\F$ be a field of characteristic $p$. Let us consider the stabilization homomorphisms
$$
\stab_*:H_i(\Gb{n}; \F[t]/(1+t)) \to H_i(\Gb{n+1}; \F[t]/(1+t))
$$
and
$$
\stab_*:H_i(\Gb{n}; \F[t]/(1-(-t)^d)) \to H_i(\Gb{n+1}; \F[t]/(1-(-t)^d)).
$$
\begin{enumerate}[label=\alph*), wide=0pt]
\item If $p=2$ $\stab_*$ is an epimorphism for $2i \leq n$ and an isomorphism for $2i < n$.
\item If $p>2$ $\stab_*$ is an epimorphism for $\frac{p(i-1)}{p-1}+2 \leq n$ and an isomorphism for $\frac{p(i-1)}{p-1}+2 < n$.
\item If $p=0$ $\stab_*$ is an epimorphism for $i+1 \leq n$ and an isomorphism for $i+1 < n$.
\end{enumerate}
\end{prop}
%\subsection{Stabilization an the map $\mu_*$}

The map $\mu$ commutes, up to homotopy, with the stabilization map $\stab: \confm{n} \to \confm{n+1}$:
$$
\xymatrix{
\confm{n-1} \times S^1 \ar[d]^{\stab \times \Id} \ar[r]^\mu& \confm{n} \ar[d]^\stab\\
\confm{n} \times S^1 \ar[r]^\mu & \confm{n+1}
}
$$

%\subsection{Stabilization and the map $\tau$}

The map $\tau$ naturally commutes with the stabilization homomorphism $\stab_*$ in homology, since $\tau$ is given by the multiplication by $(1-(-t)^d)/(1+t)$.

%\subsection{Stabilization and the map $J$}

We can also define a geometric stabilization map $\stab: \confmd{n}^d \to \confmd{n+1}^d$ as follows:
$$
\gstab:(P, z, y) \mapsto (P \cup \{p_{\infty}\}, z, y \sqrt[d]{z-p_{\infty}})
$$
where we set $p_\infty:= \frac{\max ( \{|p_i|, p_i \in P \} \cup \{|z|\})+1}{2}$ and since $\Re(z-p_{\infty}) < 0$ we choose
$\sqrt[d]{z-p_{\infty}}$ to be the unique $d$-th root with maximum real part among the roots with strictly positive imaginary part. %$\Im(\sqrt{z-p_{\infty}}) > 0$

The following diagram is homotopy commutative:
$$
\xymatrix{
	\conf_n \ar[d]^{\stab} \ar[r]^s& \confmd{n}^d \ar[d]^\gstab\\
	\conf_{n+1}  \ar[r]^s & \confmd{n+1}^d
}
$$
and this imply that $J$ commutes with the stabilization homomorphism $\gstab_*$.

We also need to prove that the following diagram commutes:
$$
\xymatrix{
	H_*(\Gb{n};\ring[t]/(1-(-t)^d)) \ar[d]^{\stab_*} \ar[r]^(0.63){\simeq}& H_*(\confmd{n}^d) \ar[d]^{\gstab_*}\\
	H_*(\Gb{n+1};\ring[t]/(1-(-t)^d))  \ar[r]^(0.63){\simeq} & H_*(\confmd{n+1}^d)
}
$$
This is true since the homomorphism $$\stab_*:H_*(\Gb{n};\ring[t]/(1-(-t)^d)) \to H_*(\Gb{n+1};\ring[t]/(1-(-t)^d))$$ is induced by the map $\stab: \confm{n} \to \confm{n+1}$ previously defined and it is obtained applying the Shapiro lemma to $\confm{n} = k(\Gb{n},1)$, with $\ring[t]/(1-(-t)^d)) = \ring[\Z_d] = \ring[\pi_1(\confm{n})/\pi_1(\confmd{n}^d)]$.
It is straightforward to check that the diagram
$$
\xymatrix{
 \confmd{n}^d \ar[d]^\gstab \ar[r]& 	\confm{n}\ar[d]^{\stab} \\
 \confmd{n+1}^d \ar[r] & 	\confm{n+1}
}
$$
commutes, where the horizontal maps are the usual $d$-fold coverings. As a consequence we have the following result.
\begin{lem}\label{lem:commutative_stab}
The following diagram is commutative
\begin{equation}\label{diag:stabilization}
\begin{tabular}{c}
\xymatrix @R=2pc @C=2pc {
%H_{i-1}(\Br_n; H_1(S^1 \times P)) \ar[r]^\iota \ar[d]^\simeq & H_{i-1}(\Br_n; H_1(\ddiskP))\ar[d]^\simeq  \\
H_{i-1}(\confm{n-1}) \otimes H_1(S^1) \ar[d]^{\stab_* \otimes \Id} \ar[r]^(0.6){\mu_*}  &	H_i(\confm{n}) \ar[r]^{\tau} \ar[d]^{\stab_*} &H_i(\confmd{n}^d)  \ar[r]^(0.45)J \ar[d]^{\gstab_*}& H_i(\confmd{n}^d\!\!, \conf_n) \ar[d]^{\gstab_*} \\
H_{i-1}(\confm{n}) \otimes H_1(S^1) \ar[r]^(0.6){\mu_*}  &	H_i(\confm{n+1}) \ar[r]^{\tau} &H_i(\confmd{n+1}^d)  \ar[r]^(0.45)J & H_i(\confmd{n+1}^d\!\!, \conf_{n+1}) 
}
\end{tabular}
\end{equation}
\end{lem}

%We introduce a special version of the Jacobsthal's function.
%\begin{df}
%	Let $\mathcal{P}(d)$ be the product of the odd primes that divide $d$.
%	We define $\alpha(d)$ as the length of the longest possible sequence of consecutive integers not coprime with $\mathcal{P}(d)$.
%\end{df}
%Notice that $\alpha(d) \leq \mathcal{P}(d)-\phi(\mathcal{P}(d))$. Moreover $\alpha(2^k)=0$ and for $p$ odd $\alpha(p^k) = 1$. We refer to \cite{iwaniec} for a much more refined upper bound for $\alpha(d)$.
\begin{thm}\label{thm:stabilization}
Consider homology with integer coefficients. The homomorphism 
$$
H_i(\Br_n; H_1(\surf_n^d)) \to H_i(\Br_{n+1}; H_1(\surf_{n+1}^d))
$$
is an epimorphism for $i \leq \frac{n}{2}-1 $
and an isomorphism for $i < \frac{n}{2}-1$.

Let $p$ be a prime that does not divide $d$.
For $n$ even the group $H_i(\Br_n; H_1(\surf_n^d))$ has no $p$ torsion %(for $p > 2$) 
when $\frac{pi}{p-1}+3  \leq n$ and no free part for $i+3 \leq n$. In particular for $n$ even %  or $(n,d)>1$, 
when $\frac{3i}{2}+3 \leq n$ the group $H_i(\Br_n; H_1(\surf_n^d))$ has only torsion that divides $d$. 
%\todo{la stima migliorerebbe di $1$ per $n$ dispari... lo scriviamo?}
\end{thm}
\begin{proof}
The maps in the diagram  \eqref{diag:stabilization} with $\Z_p$ coefficients fits in the map of long exact sequences
%$$
%\begin{tabular}{c}
%\xymatrix @R=1.5pc @C=0.8pc {
%\cdots \ar[r] &	H_{i-1}(\confm{n-1};\Z_p) \otimes H_1(S^1) \ar[d]^{\stab_* \otimes \Id} \ar[r]^(0.6){\iota}  &	H_i(\confmd{n}^d, \conf_n;\Z_p) \ar[d]^{\gstab_*} \ar[r] & H_{i-1} (\Br_n; H_1(\surf_n;\Z_p)) \ar[d]^{\stab_*} \ar[r] & \cdots\\
%\cdots \ar[r] &	H_{i-1}(\confm{n};\Z_p) \otimes H_1(S^1) \ar[r]^(0.55){\iota}  & H_i(\confmd{n+1}^d, \conf_{n+1};\Z_p) \ar[r]  & H_{i-1} (\Br_{n+1}; H_1(\surf_n;\Z_p)) \ar[r] & \cdots
%}
%$$
%that near $H_{i-1} (\Br_n; H_1(\surf_n;\Z_p))$
%looks as follows:
\begin{equation}\label{diag:stabilization2}
\begin{tabular}{l}
\xymatrix @R=1.5pc @C=0.8pc {
	\!\!\cdots\ar[r]^(0.2)\iota\! &	\!H_i(\confmd{n}^d\!\!, \conf_n;\Z_p)\! \ar[d]^{\gstab_*} \ar[r] &\!H_{i-1} (\Br_n; H_1(\surf_n^d;\Z_p))\! \ar[d]^{\stab_*} \ar[r] &\!H_{i-2}(\confm{n-1};\Z_p)\! \otimes\!H_1(S^1)\! \ar[d]^{\stab_* \otimes \Id} \ar[r]^(0.8)\iota &\!\cdots\!\!\\
	\!\!\cdots\!\ar[r]^(0.15)\iota\!& \!H_i(\confmd{n+1}^d\!\!, \conf_{n+1};\Z_p)\! \ar[r]  & \!H_{i-1} (\Br_{n+1}; \!H_1(\surf_n^d;\Z_p))\!\ar[r] &	\!H_{i-2}(\confm{n};\Z_p) \! \otimes\!H_1(S^1)\!\ar[r]^(0.8)\iota  & \!\cdots\!\!
}
\end{tabular}
\end{equation}
For $p=2$, from Proposition \ref{prop:stab_rank} and Lemma \ref{lem:commutative_stab} we have that the vertical map $\gstab_*$ on the left of diagram \eqref{diag:stabilization2} is an epimorphism for for $i \leq \frac{n}{2}$ and isomorphisms for $i < \frac{n}{2}$. The vertical map $\stab_* \otimes \Id$ on the right of diagram \eqref{diag:stabilization2} is an isomorphism for $i \leq \frac{n}{2}$.

This implies that $$
\stab_*:H_i(\Br_n; H_1(\surf_n^d; \Z_2)) \to H_i(\Br_{n+1}; H_1(\surf_{n+1}^d; \Z_2))
$$
is epimorphism for $i \leq \frac{n}{2} -1$
and an isomorphism for $i < \frac{n}{2}-1$. 

For $p>2$, from Proposition \ref{prop:stab_rank} and Lemma \ref{lem:commutative_stab} we have that the vertical map $\gstab_*$ on the left of diagram \eqref{diag:stabilization2} is an epimorphism for $\frac{p(i-1)}{p-1}+2 \leq n$ and isomorphisms for $\frac{p(i-1)}{p-1}+2 < n$. The vertical map $\stab_* \otimes \Id$ on the right of diagram \eqref{diag:stabilization2} is an isomorphism  for $\frac{p(i-1)}{p-1}+2 \leq n$. We notice that actually these conditions for $n$ are weaker than the condition that holds when $p=2$.

This implies that for $p>2$ $$
\stab_*:H_i(\Br_n; H_1(\surf_n^d; \Z_p)) \to H_i(\Br_{n+1}; H_1(\surf_{n+1}^d; \Z_p))
$$
is epimorphism for $\frac{pi}{p-1}+2 \leq n$
and an isomorphism for $\frac{pi}{p-1}+2 < n$. 

The same argument for $p=0$ shows that
$$
\stab_*:H_i(\Br_n; H_1(\surf_n^d; \Q)) \to H_i(\Br_{n+1}; H_1(\surf_{n+1}^d; \Q))
$$
is an epimorphism for $i+2 \leq n$
and an isomorphism for $i+2< n$. 

From the Universal Coefficients Theorem for homology we get that the homomorphism 
$$
H_i(\Br_n; H_1(\surf_n^d)) \to H_i(\Br_{n+1}; H_1(\surf^d_{n+1}))
$$
is an epimorphism for $i \leq \frac{n}{2}-1 $
and an isomorphism for $i < \frac{n}{2}-1$.

\vskip 3mm

In order to prove the second part of the theorem recall (Theorem \ref{th:no4tor}) that for $n$ odd %and $(n,d)=1$ 
the integer  homology $H_i(\Br_n; H_1(\surf_n^d))$ has only torsion that divides $d$. 
%

%For $n$ odd we have that at least one among the integers
%$$
%n-2\alpha(d), n-2(\alpha(d)+1), \ldots, n
%$$
%is co-prime with $d$.

The stabilization implies that if $p \nmid d$, for $n$ even $H_i(\Br_n; H_1(\surf_n^d))$ has no $p$ torsion %(for $p > 2$) 
for $\frac{pi}{p-1}+3  \leq n$ and no free part for $i+3 \leq n$ and 
for $n$ odd, $H_i(\Br_n; H_1(\surf_n^d))$ has no $p$ torsion %(for $p > 2$) 
for $\frac{pi}{p-1}+2  \leq n$ and no free part for $i+2 \leq n$

In particular, for $n$ even,  $\frac{3i}{2}+3  \leq n$ we have that $H_i(\Br_n; H_1(\surf_n^d))$ has only torsion that divides $d$.
\end{proof}
The following statement is the analog of the result in \cite[Prop.~9]{chen} for cohomology.
\begin{thm}\label{thm:unstable}
For $n$ even the groups $H_i(\Br_n;H_1(\surf_n^d;\Z))$ are torsion, except when $d$ is even and $i=n-1, n-2$ where $H_i(\Br_n;H_1(\surf_n^d;\Q)) = \Q.$
\end{thm}
\begin{proof}
The result follows from Proposition \ref{prop:coker_n_pari} and, for $i<2$, from the stabilization Theorem \ref{thm:stabilization}. For $n=4$ and $i<2$ the result follows from a direct computation (see also 
%\cite[Prop.~9]{chen} or 
\cite[Thm.~6.4]{calmar}).
\end{proof}

Since from Theorem \ref{thm:Bddn} we have the isomorphism $$H_i(\B(d,d,n)) \simeq H_i(\Br_n) \oplus H_{i-1}(\Br_n;H_1(\surf_n^d;\Z))$$
the result of Theorem \ref{thm:stabilization} above and the stability for the homology of the classical braid groups (see \cite{arnold70}) imply a stabilization of the homology of the complex braid groups of type $\B(d,d,n)$.
\begin{thm}\label{thm:complexbraidstability}
The homomorphism $$
H_i(\B(d,d,n)) \to H_{i}(\B(d,d,n+1))
$$
induced by the natural inclusion 
$\B(d,d,n) \into \B(d,d,n+1)$ 
is an epimorphism for $i \leq \frac{n}{2} $
and an isomorphism for $i < \frac{n}{2}$.
\end{thm}

In \Cref{tab:conti,tab:conti2,tab:conti3,tab:conti4,tab:conti5} we  present some computations of the groups $H_i(\Br_n;H_1(\surf_n^d))$ for $d = 2$ to $6$, with integer coefficients. The computations are obtained using an Axiom implementation of the complex introduced in \cite{salvetti}. Notice that the homology groups in the Table \ref{tab:conti} coincide with the homology groups  $H_{i+1}(\Gdn, \Br_n)$ (see Theorem \ref{thm:Bddn} and Remark \ref{rem:dn}).
\begin{table*}[htb] 
\setlength{\tabcolsep}{6pt}
\renewcommand{\arraystretch}{1.2}
\setlength\extrarowheight{3pt}
\begin{center}
\begin{tabular}{|c|c|c|c|c|c|c|c|c|c|c|c|}
\hline
\backslashbox{$n$}{$i$}
& 1 & 2 & 3 & 4 & 5 & 6 & 7 & 8 & 9 & 10 & 11\\
\hline
$3$ & $\Z_2$ &&&&&&&&&&\\
\hline
$4$ & $\Z_2^2$ & $\Z$&$\Z$&&&&&&&&\\
\hline
$5$
 & 
\cellcolor{lightgray!30}
$\Z_2$ & 
$\Z_2$
&$\Z_2$&&&&&&&&\\ 
\hline
$6$ & $\Z_2$ & $\Z_2^2$&\!$\Z_2^2  \Z_3$\!&$\Z$&$\Z$&&&&&&\\
\hline
$7$ & $\Z_2$ & \chl $\Z_2$&$\Z_2^2 $&$\Z_2^2$&$\Z_2$&&&&&&\\
\hline
$8$ & $\Z_2$ & $\Z_2$&\!\!$\Z_2^3 $\!\!&\!\!$\Z_2^3\Z_3$\!\!&$\Z_2^3\Z_3$&$\Z$&$\Z$&&&&\\
\hline
$9$ & $\Z_2$ & $\Z_2$&\chl $\Z_2^2 $&$\Z_2^3$&$\Z_2^3$&$\Z_2^2$&$\Z_2$&&&&\\
\hline
${10}$ & $\Z_2$ & $\Z_2$&$\Z_2^2 $&$\Z_2^4$&$\Z_2^4$&$\Z_2^4\Z_3$&\!$\Z_2^3 \Z_3\Z_5$\!&$\Z$&$\Z$&&\\
\hline
${11}$ & $\Z_2$ & $\Z_2$&$\Z_2^2 $&\chl $\Z_2^3$&$\Z_2^4$&$\Z_2^4$&$\Z_2^4 $&$\Z_2^3$&$\Z_2$&&\\
\hline
${12}$ & $\Z_2$ & $\Z_2$&$\Z_2^2 $&$\Z_2^3$&
% $\Z_2^5$&$\Z_2^5$ 
$\Z_2^5$ & $\Z_2^5$
&$\Z_2^6\Z_3$&\!$\Z_2^6\Z_3\Z_5$\!&$\Z_2^3\Z_3\Z_5$&$\Z$&$\Z$\\
\hline
${13}$ & $\Z_2$ & $\Z_2$&$\Z_2^2 $&$\Z_2^3$&\chl $\Z_2^4$&$\Z_2^5$&$\Z_2^6$&$\Z_2^6$&$\Z_2^5$&$\Z_2^3$&$\Z_2$\\
\hline
\end{tabular}
\end{center}
\caption{Computations of $H_i(\Br_n;H_1(\surf_n^2))$. For each column the first stable group is highlighted.}\label{tab:conti}
\end{table*}

\begin{table*}[htb] 
	\setlength{\tabcolsep}{6pt}
	\renewcommand{\arraystretch}{1.2}
	\setlength\extrarowheight{3pt}
	\begin{center}
		\begin{tabular}{|c|c|c|c|c|c|c|c|c|c|c|c|}
			\hline
			\backslashbox{$n$}{$i$}
			& 1 & 2 & 3 & 4 & 5 & 6 & 7 & 8 \\
			\hline
			$3$ & $\Z_3$ &&&&&&&\\
			\hline
			$4$ & 	\chl	 $\Z_3$ & $\Z_3$&&&&&&\\
			\hline
			$5$
			& 

			$\Z_3$ & 
			$\Z_3$
			&$\Z_3$&&&&&\\ 
			\hline
			$6$ & $\Z_3$ & \chl $\Z_3$&$\Z_3^2$  &$\Z_3^2$&&&&\\
			\hline
			$7$ & $\Z_3$ &  $\Z_3$&\chl $\Z_3 $&$\Z_3$&$\Z_3$&&&\\
			\hline
			$8$ & $\Z_3$ & $\Z_3$&$\Z_3 $&$\Z_3$&$\Z_3$&$\Z_3$&&\\
			\hline
			$9$ & $\Z_3$ & $\Z_3$& $\Z_3 $&\chl $\Z_3$ &$\Z_3^2$ &$\Z_3^2$&$\Z_3$&\\
			\hline
			$10$ & $\Z_3$ & $\Z_3$& $\Z_3 $& $\Z_3$ &$\Z_3^2$ &$\Z_3^3$&$\Z_3^2$&$\Z_3$\\
			\hline
%			${10}$ & $\Z_2$ & $\Z_2$&$\Z_2^2 $&$\Z_2^4$&$\Z_2^4$&$\Z_2^4\Z_3$&$\Z_2^3 \Z_3\Z_5$&$\Z$&$\Z$&&\\
%			\hline
%			${11}$ & $\Z_2$ & $\Z_2$&$\Z_2^2 $&\chl $\Z_2^3$&$\Z_2^4$&$\Z_2^4$&$\Z_2^4 $&$\Z_2^3$&$\Z_2$&&\\
%	\hline
%			${12}$ & $\Z_2$ & $\Z_2$&$\Z_2^2 $&$\Z_2^3$&			$\Z_2^5$ & $\Z_2^5$			&$\Z_2^6\Z_3$&$\Z_2^6\Z_3\Z_5$&$\Z_2^3\Z_3\Z_5$&$\Z$&$\Z$\\
%			\hline
%			${13}$ & $\Z_2$ & $\Z_2$&$\Z_2^2 $&$\Z_2^3$&\chl $\Z_2^4$&$\Z_2^5$&$\Z_2^6$&$\Z_2^6$&$\Z_2^5$&$\Z_2^3$&$\Z_2$\\
%			\hline
		\end{tabular}
	\end{center}
	\caption{Computations of $H_i(\Br_n;H_1(\surf_n^3))$. For each column the first stable group is highlighted.}\label{tab:conti2}
\end{table*}
\begin{table*}[htb] 
	\setlength{\tabcolsep}{6pt}
	\renewcommand{\arraystretch}{1.2}
	\setlength\extrarowheight{3pt}
	\begin{center}
		\begin{tabular}{|c|c|c|c|c|c|c|c|c|c|c|c|}
			\hline
			\backslashbox{$n$}{$i$}
			& 1 & 2 & 3 & 4 & 5 & 6 & 7 & 8 & 9\\
			\hline
			$3$ & $\Z_4$ &&&&&&&&\\
			\hline
			$4$ & 	\!\!	 $\Z_2 \Z_4$\!\! & \!\!$\Z_2^2 \Z$ \!\!&$\Z$&&&&&&\\
			\hline
			$5$
			& 
			\chl$\Z_4$ & 
			$\Z_4$
			&$\Z_4$&&&&&&\\ 
			\hline
			$6$ & $\Z_4$ &\!\! $\Z_2 \Z_4 $\!\!&$\Z_2\Z_3 \Z_4$  &$\Z_2\Z $&$\Z$&&&&\\
			\hline
			$7$ & $\Z_4$ &  \chl $\Z_4$& $\Z_2\Z_4 $&$\Z_2\Z_4$&$\Z_4$&&&&\\
			\hline
			$8$ & $\Z_4$ & $\Z_4$&$\Z_2^2\Z_4 $&$\Z_2^3\Z_3\Z_4$&$\Z_2^3\Z_3\Z_4$&$\Z_4\Z_8\Z$&$\Z$&&\\
			\hline
			$9$ & $\Z_4$ & $\Z_4$& \chl$\Z_2\Z_4 $& $\Z_2^2\Z_4$ &$\Z_2^2\Z_4$ &$\Z_2\Z_4$&$\Z_4$&&\\
			\hline
			${10}$ & $\Z_4$ & $\Z_4$&$\Z_2 \Z_4$&$\Z_2^3 \Z_4$&$\Z_2^3 \Z_4$&$\Z_2^3\Z_4\Z_6$&$\Z_2^3 \Z_{60}$&$\Z_2\Z$&$\Z$\\
						\hline
			%			${11}$ & $\Z_2$ & $\Z_2$&$\Z_2^2 $&\chl $\Z_2^3$&$\Z_2^4$&$\Z_2^4$&$\Z_2^4 $&$\Z_2^3$&$\Z_2$&&\\
			%	\hline
			%			${12}$ & $\Z_2$ & $\Z_2$&$\Z_2^2 $&$\Z_2^3$&			$\Z_2^5$ & $\Z_2^5$			&$\Z_2^6\Z_3$&$\Z_2^6\Z_3\Z_5$&$\Z_2^3\Z_3\Z_5$&$\Z$&$\Z$\\
			%			\hline
			%			${13}$ & $\Z_2$ & $\Z_2$&$\Z_2^2 $&$\Z_2^3$&\chl $\Z_2^4$&$\Z_2^5$&$\Z_2^6$&$\Z_2^6$&$\Z_2^5$&$\Z_2^3$&$\Z_2$\\
			%			\hline
		\end{tabular}
	\end{center}
	\caption{Computations of $H_i(\Br_n;H_1(\surf_n^4))$. For each column the first stable group is highlighted. Notice that $3$-torsion appears for $n$ even.}\label{tab:conti3}
\end{table*}

\begin{table*}[htb] 
	\setlength{\tabcolsep}{6pt}
	\renewcommand{\arraystretch}{1.2}
	\setlength\extrarowheight{3pt}
	\begin{center}
		\begin{tabular}{|c|c|c|c|c|c|c|c|c|c|c|c|}
			\hline
			\backslashbox{$n$}{$i$}
			& 1 & 2 & 3 & 4 & 5 & 6 & 7 & 8 & 9 & 10 & 11\\
			\hline
			$3$ & $\Z_5$ &&&&&&&&&&\\
			\hline
			$4$ & 		\chl $\Z_5 $ & $\Z_5$&&&&&&&&&\\
			\hline
			$5$
			& 
			$\Z_
			5$ & \chl
			$\Z_5$
			& $\Z_5$&&&&&&&&\\ 
			\hline
			$6$ & $\Z_5$ & $\Z_5$&\chl $\Z_5$  &$\Z_5$&&&&&&&\\
			\hline
			$7$ & $\Z_5$ &  $\Z_5$& $\Z_5 $&$\Z_5$&$\Z_5$&&&&&&\\
			\hline
			$8$ & $\Z_5$ & $\Z_5$&$\Z_5$&\chl $\Z_5$&$\Z_5$&$\Z_5$&&&&&\\
			\hline
			$9$ & $\Z_5$ & $\Z_5$& $\Z_5 $& $\Z_5$ &\chl $\Z_5$ &$\Z_5$&$\Z_5$&&&&\\
			\hline
			$10$ & $\Z_5$ & $\Z_5$& $\Z_5 $& $\Z_5$ & $\Z_5$ &$\Z_5$&$\Z_5^4$&$\Z_5^4$&&&\\
			\hline
			%			${10}$ & $\Z_2$ & $\Z_2$&$\Z_2^2 $&$\Z_2^4$&$\Z_2^4$&$\Z_2^4\Z_3$&$\Z_2^3 \Z_3\Z_5$&$\Z$&$\Z$&&\\
			%			\hline
			%			${11}$ & $\Z_2$ & $\Z_2$&$\Z_2^2 $&\chl $\Z_2^3$&$\Z_2^4$&$\Z_2^4$&$\Z_2^4 $&$\Z_2^3$&$\Z_2$&&\\
			%	\hline
			%			${12}$ & $\Z_2$ & $\Z_2$&$\Z_2^2 $&$\Z_2^3$&			$\Z_2^5$ & $\Z_2^5$			&$\Z_2^6\Z_3$&$\Z_2^6\Z_3\Z_5$&$\Z_2^3\Z_3\Z_5$&$\Z$&$\Z$\\
			%			\hline
			%			${13}$ & $\Z_2$ & $\Z_2$&$\Z_2^2 $&$\Z_2^3$&\chl $\Z_2^4$&$\Z_2^5$&$\Z_2^6$&$\Z_2^6$&$\Z_2^5$&$\Z_2^3$&$\Z_2$\\
			%			\hline
		\end{tabular}
	\end{center}
	\caption{Computations of $H_i(\Br_n;H_1(\surf_n^5))$. For each column the first stable group is highlighted.}\label{tab:conti4}
\end{table*}

\begin{table*}[htb] 
	\setlength{\tabcolsep}{6pt}
	\renewcommand{\arraystretch}{1.2}
	\setlength\extrarowheight{3pt}
	\begin{center}
		\begin{tabular}{|c|c|c|c|c|c|c|c|c|c|c|c|}
			\hline
			\backslashbox{$n$}{$i$}
			& 1 & 2 & 3 & 4 & 5 & 6 & 7 & 8 & 9 \\
			\hline
			$3$ & $\Z_6$ &&&&&&&&\\
			\hline
			$4$ & 		\! $\Z_2 \Z_6$\! & $\Z_3 \Z$&$\Z$&&&&&&\\
			\hline
			$5$
			& \chl
			$\Z_
			6$ & 
			$\Z_6$
			&$\Z_6$&&&&&&\\ 
			\hline
			$6$ & $\Z_6$ & \!$\Z_2 \Z_6$\!&\!$\Z_3^2 \Z_6^2$ \! &$\Z_3^4 \Z$&$\Z$&&&&\\
			\hline
			$7$ & $\Z_6$ &  \chl$\Z_6$& \!$\Z_2 \Z_6 $\!&\!$\Z_2\Z_6$\!&$\Z_6$&&&&\\
			\hline
			$8$ & $\Z_6$ & $\Z_6$&\!$\Z_2^2\Z_6 $\!&\!$\Z_2\Z_6^2$\!&\!$\Z_2\Z_6^2$\!&$\Z_3\Z$&$	\Z$&&\\
			\hline
			$9$ & $\Z_6$ & $\Z_6$&  \chl \!$\Z_2 \Z_6$\!  & \!$\Z_2^2\Z_6$  \!&\! $\Z_2\Z_6^2$ \!&$\Z_6^2$&$\Z_6$&&\\
			\hline
			$10$ & $\Z_6$ & $\Z_6$& \! $\Z_2 \Z_6$ \! & ?  & $\Z_2^2 \Z_6^2$ &$\Z_6^4$&\!$\Z_6^3\Z_5$\!&\!$\Z_3\Z$\! & $\Z$\\
			\hline
			
			%			${10}$ & $\Z_2$ & $\Z_2$&$\Z_2^2 $&$\Z_2^4$&$\Z_2^4$&$\Z_2^4\Z_3$&$\Z_2^3 \Z_3\Z_5$&$\Z$&$\Z$&&\\
			%			\hline
			%			${11}$ & $\Z_2$ & $\Z_2$&$\Z_2^2 $&\chl $\Z_2^3$&$\Z_2^4$&$\Z_2^4$&$\Z_2^4 $&$\Z_2^3$&$\Z_2$&&\\
			%	\hline
			%			${12}$ & $\Z_2$ & $\Z_2$&$\Z_2^2 $&$\Z_2^3$&			$\Z_2^5$ & $\Z_2^5$			&$\Z_2^6\Z_3$&$\Z_2^6\Z_3\Z_5$&$\Z_2^3\Z_3\Z_5$&$\Z$&$\Z$\\
			%			\hline
			%			${13}$ & $\Z_2$ & $\Z_2$&$\Z_2^2 $&$\Z_2^3$&\chl $\Z_2^4$&$\Z_2^5$&$\Z_2^6$&$\Z_2^6$&$\Z_2^5$&$\Z_2^3$&$\Z_2$\\
			%			\hline
		\end{tabular}
	\end{center}
	\caption{Computations of $H_i(\Br_n;H_1(\surf_n^6))$. For each column the first stable group is highlighted. Notice that $5$-torsion appears for $n$ even.}\label{tab:conti5}
\end{table*}

\subsection{Poincar\'e polynomials}
We use the previous results to compute explicitly the Poincar\'e polynomials for odd $n$.%, $(n,d)=1$.
%\todo{rivedere...}
%-ogni generatore da luogo ad una classe stabile. infatti la stabilizzazione ($n \mapsto n+2$ ) viene data dalla moltiplicazione per $x_0^2$ che commuta con $\tau$.

\begin{thm}\label{thm:poincare}
For odd $n$, for $p$ prime such that $p\mid d$ 
the rank of $H_i(\Br_n;H_1(\surf_n^d;\Z))\otimes \Z_p$ as a $\Z_p$-module is the coefficient of $q^it^n$ in the expansion of the series
$$
\widetilde{P}_p(q,t)= \frac{qt^3}{(1-t^2q^2)(1-t^2)} \prod_{j \geq 0} \frac{1+q^{2p^j-1}t^{2p^j}}{1-q^{2p^{j+1}-2}t^{2p^{j+1}}}
$$
that specialize in the case  $p=2$ to the series
 $$
\widetilde{P}_2(q,t)=\frac{qt^3}{(1-t^2q^2)} \prod_{i \geq 0} \frac{1}{1-q^{2^i-1}t^{2^i}}.
$$
%for $p=2$ and
% 
%for $p$ odd.
%In particular the series $\widetilde{P}_2(q,t)$ is 
%the Poincar\'e series of the group  %$$\oplus_{n}H_*(\Br_{2n+1};H_1(\surf_{2n+1};\Z))$$ 
%$$\bigoplus_{n \mbox{\scriptsize odd}}H_*(\Br_{n};H_1(\surf^d_n;\Z)) \otimes \Z_p$$
%as a $\Z_2$-module. 
%is the coefficient of $t^n$ of the series:	
\end{thm}

\begin{proof}
Let $P_p(\Br_n,H_1(\surf_n^d))(q)$ be the Poincar\'e polynomial for the homology group $H_*(\Br_n;H_1(\surf_n^d;\Z))\otimes \Z_p$ as a $\Z_p$-module. 
Since we already know that for $n$ odd %$(n,d)=1$ 
the homology group $H_i(\Br_n;H_1(\surf_n^d;\Z))$ has only torsion of order that divides $d$, we can compute the polynomial $P_p(\Br_n,H_1(\surf_n^d))(q)$ from the Universal Coefficients Theorem as follows. We compute the Poincar\'e polynomial 
$
P_p(\Br_n, H_1(\surf_n^d;\Z_p))(q)
$
for the group 
$H_i(\Br_n;H_1(\surf_n^d;\Z_p)) \otimes \Z_p
$
and we divide by $1+q$.

In order to compute $
P_p(\Br_n, H_i(\surf_n^d;\Z_p))(q)
$ 
we consider the short exact sequence, which splits
$$
0 \to \coker \iota_i \to H_{i}(\Br_n; H_1(\surf_n^d:\Z_p)) \to \ker \iota_{i-1} \to 0
$$
where we recall that $\iota_i$ is the map
$$\iota_i:H_{i}(\Br_n; H_1(S^1 \times P;\Z_p)) \to H_{i}(\Br_n; H_1(\ddiskP^d;\Z_p)).$$
It follows from the Remark \ref{rem:basi_mod_2} and from the description in Section \ref{sec:no4tor} that for a fixed odd $n$ the ranks of $\ker \iota$ and $\coker \iota$
are respectively the coefficients of $q^it^n$ in the following series:
$$
P_p(\coker \iota) = P_p(\ker \iota) = \frac{qt^3}{(1-t^2q^2)(1-t^2)} \prod_{j \geq 0} \frac{1+q^{2p^j-1}t^{2p^j}}{1-q^{2p^{j+1}-2}t^{2p^{j+1}}}
$$
that specalize for $p=2$ to the series
$$
P_2(\coker \iota) = P_2(\ker \iota)  =
\frac{qt^3}{1-t^2q^2} \prod_{j \geq 0} \frac{1}{1-q^{2^j-1}t^{2^j}}
$$
%and for $p$ odd

% 	\todo{da controllare}
Clearly 
%in  both cases $p=2$ and $p$ odd 
the polynomial for 
$H_{i}(\Br_n; H_1(\surf_n^d:\Z_p)) \otimes \Z_p$ is given by the coefficient of $t^n$ in the sum
$$
 P_p(\coker \iota) + qP_p(\ker \iota) 
$$
and hence dividing by $(1+q)$ we get our result. 
\end{proof}

\begin{rem}
The same argument of Theorem \ref{thm:poincare} could be used when $d$ is odd  to compute a (more complicated) formula of the Poincar\'e polynomial of the groups $H_*(\Br_n;H_1(\surf_n^d;\Z))\otimes \Z_p$ for any $n$.
\end{rem}

\begin{rem}
Notice that the polynomials given in Theorem \ref{thm:poincare} do not depend on the integer $d$. Hence we have that when $n$ is odd and $p \mid d$ there is a (non-natural) isomorphism$$
H_i(\Br_n;H_1(\surf_n^d;\Z))\otimes \Z_p = H_i(\Br_n;H_1(\surf_n^p;\Z))\otimes \Z_p.
$$
\end{rem}

\begin{rem}
We know (Theorem \ref{th:only_d_torsion} and \ref{thm:4tor}) that when $n$ is odd and $d$ is squarefree
the group
$H_i(\Br_n; H_1(\surf_n^d;\Z))$ is finite (with the exception of $H_0$) and has no $p$-torsion for $p \nmid d$ and no $p^2$-torsion for any prime $p$. Then in such cases Theorem \ref{thm:poincare} completely determines the homology groups $H_i(\Br_n; H_1(\surf_n^d;\Z))$ and hence (using Theorem \ref{thm:Bddn}) the groups $H_{i+1}(\B(d,d,n), \Br_n;\Z)$ for $n$ odd and $d$ squarefree.
\end{rem}
%\begin{cor}
The same argument of the previous proof can be applied in stable rank. From the Remark \ref{rem:basi_mod_2} the Stable Poincar\'e polynomial of both $\coker \iota$ and $\ker \iota$ with $\Z_p$ coefficients is the following:
$$
\frac{q}{1-q^2} \prod_{j \geq 0} \frac{1+q^{2p^j-1}}{1-q^{2p^{j+1}-2}}
$$
and in particular for $p=2$ we obtain
$$
\frac{q}{1-q^2} \prod_{j \geq 1} \frac{1}{1-q^{2^j-1}}.
$$

Since for $n$ odd there is no free part in $H_i(\Br_n;H_1(\surf_n^d;\Z))$ all these groups have only torsion that divides $d$. In particular for integer coefficients we get the following statement.
\begin{thm}\label{thm:stablepoincare}
Let $p$ be a
%n odd 
prime and let $d$ be an integer such that $p\mid d$ the Poincar\'e polynomial of the stable homology $H_i(\Br_n;H_1(\surf_n^d;\Z))\otimes \Z_p$ as a $\Z_p$-module is the following:
$$
P_p(\Br;H_1(\Sigma^d))(q) = \frac{q}{1-q^2} \prod_{j \geq 0} \frac{1+q^{2p^j-1}}{1-q^{2p^{j+1}-2}}.
$$
In particular when $d$ is an even integer the Poincar\'e polynomial of the stable homology $H_i(\Br_n;H_1(\surf_n^d;\Z))\otimes \Z_2$ as a $\Z_2$-module is the following:
$$
P_2(\Br;H_1(\Sigma^d))(q) = \frac{q}{1-q^2} \prod_{j \geq 1} \frac{1}{1-q^{2^j-1}}.
$$
\end{thm}
\begin{rem}
The stabilization result for the homology of the complex braid groups of type $\B(d,d,n)$ (Theorem \ref{thm:complexbraidstability}) together with Theorem \ref{th:no4tor} implies that the stable homology of $\B(d,d,n)$ with integer coefficients has no $p^k$ torsion only if $p^k \nmid d$.
Hence Theorem \ref{thm:Bddn} implies that when $p^2 \nmid d$ the stable Poincar\'e polynomials given in Theorem \ref{thm:stablepoincare} determines the corresponding $p$-torsion component of the stable homology group $H_{i+1}(\B(d,d,n), \Br_n; \Z)$ and this component 
is the same for all $d = p m$ where $p \nmid m$  and is trivial when $p \nmid d$.
%does not depend on $d$.
\end{rem}
An explicit computation of the first terms of the stable series $P_2(\Br;H_1(\Sigma^{2d}))(q)$ gives
$$
q+q^2+ 2q^3 + 3q^4 + 4q^5 + 5q^6 + 7q^7 + 9 q^8 + 11q^9 + 14q^{10} + 17 q^{11} + \ldots
$$
while for $P_3(\Br_n;H_1(\surf_n^{3d};\Z))$ we have
$$
q + q^2 + q^3 + q^4 + 2q^5 + 3q^6 + 3 q^7 + 3q^8 + 4q^9 + 5q^{10} + 5q^{11} + 6 q^{12} + \ldots
$$

%\section{More}
%-generalizzazione per rivestimenti a $p$ fogli?

%\addcontentsline{toc}{section}{References}
%\nocite{*} % comment this to remove uncited references
\bibliographystyle{alpha}
\bibliography{biblio} 
%\affiliationone{% in this example, two authors share an institution
%Filippo Callegaro and Mario Salvetti\\
%Dipartimento di Matematica, \\
%University of Pisa \\
%Largo Bruno Pontecorvo, 5\\
%57127 Pisa\\
%Italy
%\email{callegaro@dm.unipi.it\\
%salvetti@dm.unipi.it}}
\end{document}